\documentclass[a4paper,11pt]{amsart}
\usepackage{amsmath,amsthm,amssymb,array,xcolor,booktabs,hyperref,tikz,comment,stmaryrd,enumitem}
\usepackage[margin=2.5cm]{geometry}
\usepackage[all]{xy}
\usepackage{mathtools}
\usetikzlibrary{cd}

\newcommand{\gp}{\Gamma}
\newcommand{\lat}[1][\gp]{\Lambda_{#1}}
\newcommand{\hol}[1][\gp]{H_{#1}}
\newcommand{\RR}{\mathbb{R}}
\newcommand{\CC}{\mathbb{C}}
\newcommand{\HH}{\mathbb{H}}
\newcommand{\ZZ}{\mathbb{Z}}
\newcommand{\QQ}{\mathbb{Q}}
\newcommand{\KK}{\mathbb{K}}
\DeclareMathOperator\Aut{Aut}
\DeclareMathOperator\Aff{Aff}
\DeclareMathOperator\AffDiff{AffDiff}
\DeclareMathOperator\Iso{Iso}
\DeclareMathOperator\End{End}
\DeclareMathOperator\Fix{Fix}
\DeclareMathOperator\GLin{GL}
\DeclareMathOperator\Ort{O}
\DeclareMathOperator\Mat{Mat}
\DeclareMathOperator\Gal{Gal}
\DeclareMathOperator\Diff{Diff}
\DeclareMathOperator\Out{Out}
\DeclareMathOperator\Id{Id}

\DeclareMathOperator\Tr{Tr}
\DeclareMathOperator\MCG{MCG}
\DeclareMathOperator\Hom{Hom}
\newcommand{\GL}[2][n]{\GLin_{#1}(#2)}
\newcommand{\OO}[2][n]{\Ort_{#1}(#2)}

\newcommand{\Mflat}[1][\mathcal{O}]{\mathcal{M}_{\textup{flat}}(#1)}
\newcommand{\Tflat}[1][\mathcal{O}]{\mathcal{T}_{\textup{flat}}(#1)}
\newcommand{\Rflat}[1][\mathcal{O}]{\mathcal{R}_{\textup{flat}}(#1)}
\newcommand{\nmlr}[2][\hol]{\mathrm{N}_{#2}\mathopen{}\left(#1\right)\mathclose{}}
\newcommand{\ctlr}[2][\hol]{\mathrm{C}_{#2}\mathopen{}\left(#1\right)\mathclose{}}

\newcommand{\tNG}{\mathcal{N}_\Gamma}

\newtheorem{thm}{Theorem}[section]
\newtheorem{prop}[thm]{Proposition}
\newtheorem{lem}[thm]{Lemma}
\newtheorem{cor}[thm]{Corollary}

\newtheorem{thmx}{Theorem}
\newtheorem{corx}[thmx]{Corollary}

\theoremstyle{definition}
\newtheorem{definition}[thm]{Definition}
\newtheorem{sta}[thm]{Standing Assumption}

\theoremstyle{remark}
\newtheorem{rmk}[thm]{Remark}

\title{Moduli spaces of flat Riemannian metrics on orbifolds}

\author{Karla Garc\'ia}
\address[K. Garc\'ia]{
Departamento de Matem\'aticas, Facultad de Ciencias, Universidad Nacional Autónoma de M\'exico, Circuito Exterior, Ciudad Universitaria, 04510 Coyoac\'an, Ciudad de M\'exico, Mexico}
\email{ohmu@ciencias.unam.mx}

\author{Ingrid Amaranta Membrillo Solis}
\address[I. A. Membrillo Solis]{School of Mathematical Sciences, Queen Mary University of London,  327 Mile End Road, London E1 4NS, United Kingdom}
\email{i.a.membrillosolis@qmul.ac.uk}

\author{Motiejus Valiunas}
\address[M. Valiunas]{Instytut Matematyczny, Uniwersytet Wroc{\l}awski, plac Grunwaldzki 2, 50-384 Wroc{\l}aw, Poland}
\email{motiejus.valiunas@math.uni.wroc.pl}

\date{}

\setenumerate{label=\textup{(\roman*)}}

\begin{document}

\begin{abstract}
    We study moduli spaces of flat metrics on closed Riemannian orbifolds admitting such metrics.  We show that for such orbifolds $\mathcal{O}$, the Teichm\"uller space of flat metrics $\Tflat$ serves as a classifying space for proper actions of the mapping class group, $\MCG(\mathcal{O})$, in the appropriate category.  We show that the moduli space of flat metrics, $\Mflat = \Tflat / \MCG(\mathcal{O})$, is itself a very good orbifold, and, under a technical assumption, describe this orbifold algebraically up to commensurability.
\end{abstract}
\maketitle

\section{Introduction}

Let $S_g$ be a surface of finite genus $g$, possibly with punctures and non-empty boundary.  When $S_g$ admits a hyperbolic metric, that is, a Riemannian metric of constant negative curvature, the rich geometry of such structures gives rise to the classical Teichm\"uller space $\mathcal{T}_{\mathrm{hyp}}(S_g)$, which parametrizes hyperbolic metrics on $S_g$ up to isotopy. The mapping class group $\MCG(S_g)$, defined as the group of isotopy classes of self-diffeomorphisms of $S_g$, acts naturally on $\mathcal{T}_{\mathrm{hyp}}(S_g)$, and the quotient is the moduli space of hyperbolic structures. These spaces have been extensively studied due to their deep connections to low-dimensional topology, geometric group theory, and algebraic geometry; see, for example, \cite{Farb-Margalit, imayoshi2012, hubbard2016} and the references therein.

A key advantage of the aforementioned spaces is that they admit rich algebraic structures, for example, structures of complex analytic or algebraic varieties. In this work, we consider an analogous setting for \emph{flat metrics}, that is, Riemannian metrics with zero sectional curvature. It turns out that the spaces of such metrics, when defined appropriately, admit a structure similar in spirit to that of the hyperbolic case. However, it is important to note that the techniques used to study spaces of hyperbolic metrics are highly specific to two-dimensional surfaces. In contrast, the study of spaces of negatively curved metrics on higher-dimensional manifolds involves substantially different tools and ideas; for an introduction, see~\cite{Tuschmann-Wraith}.

The motivation behind this work is to construct and analyze analogues of Teichm\"uller and moduli spaces in the flat setting. One of the main benefits of working with flat metrics is that the associated mapping class groups, as well as the corresponding Teichm\"uller and moduli spaces, can be naturally defined not only for closed flat manifolds but also for closed flat orbifolds. These orbifolds, which generalize manifolds by allowing controlled singularities, arise naturally in the context of flat geometry and broaden the scope of the theory.

Locally modeled on quotients of Euclidean space by finite group actions, orbifolds preserve much of the analytic and geometric structure of manifolds. They were first introduced by Satake \cite{Satake1956}, and later developed further by Thurston \cite{thurston2022} in the context of 3-manifold topology and geometric structures. Orbifolds arise naturally in several areas of mathematics and physics, including string theory, crystallography, and equivariant topology.  In the Riemannian setting, orbifolds support much of the analytic structure familiar from manifolds. Classical tools such as elliptic regularity, Hodge theory, and index theory extend naturally to orbifolds \cite{Kawasaki1979, kawasaki1981}. Heat kernel asymptotics and spectral theory have been developed in this context~\cite{dryden2008, gittins2024}, and geometric flows such as Ricci flow have been successfully adapted to orbifolds~\cite{kleiner2011}.

Motivated by this broad framework, in this paper we study closed flat orbifolds $\mathcal{O}$, as well as the Teichm\"uller and moduli spaces of flat metrics, and the mapping class groups associated with them. The geometric framework that we use is the best to generalize the notion of moduli space of flat metrics over closed manifolds to closed orbifolds. It follows from the Bieberbach Theorems (see Theorem~\ref{thm:bieb}) that such orbifolds of dimension~$n$ can be described as quotients $\mathcal{O} = \Gamma \backslash \RR^n$, where $\Gamma$ is a crystallographic group which is a discrete cocompact subgroup of the group $\Aff(n)$ of affine linear transformations of $\RR^n$.  We have $\Aff(n) \cong \RR^n \rtimes \GL{\RR}$, and we write $\tau\colon \Aff(n) \to \GL{\RR}$ for the canonical projection map.  We call $\hol \coloneqq \tau(\Gamma)$ and $\lat \coloneqq \ker(\tau) \cap \Gamma \leq \ker(\tau) \cong \RR^n$ the \emph{holonomy} and the \emph{translation lattice} of $\Gamma$, respectively.

For a closed flat orbifold $\mathcal{O}$, we consider its mapping class group $\MCG(\mathcal{O})$ in the category of affine diffeomorphisms, as well as its Teichm\"uller space of flat metrics, $\Tflat$ (see Sections \ref{ssec:prelims-modsp} and \ref{ssec:prelims-mcg} for details). Our first result gives an algebraic description of $\MCG(\mathcal{O})$ up to taking quotients by finite subgroups, and shows that $\Tflat$ is a classifying space for the family of finite subgroups of $\MCG(\mathcal{O})$.  The latter result is motivated by a broader interest in equivariant topology and geometric group theory, where classifying spaces for families of subgroups play a fundamental role. Given a group $G$ and a family $\mathcal{F}$ of subgroups (such as the family of finite subgroups), a classifying space $E_{\mathcal{F}}(G)$ is a $G$-CW complex whose isotropy groups lie in $\mathcal{F}$ and which is terminal among such spaces up to $G$-homotopy equivalence. These spaces generalize the universal space $EG$ for free $G$-actions and are central tools in the computation of group homology and algebraic $K$- and $L$-theory, particularly in relation to the Farrell--Jones conjecture (see \cite{luck2005, luck-reich2005}).

In our context, identifying $\Tflat$ as a model for $E_{\mathcal{F}_{\text{fin}}}(\MCG(\mathcal{O}))$ has both topological and algebraic implications. It enables the study of $\MCG(\mathcal{O})$-equivariant structures on moduli spaces and relates geometric structures on $\mathcal{O}$ to homotopy-theoretic invariants. Analogous results are known in the surface case: Kerckhoff \cite{Kerckhoff} showed that the Teichm\"uller space $\mathcal{T}_{\text{hyp}}(S_g)$ of hyperbolic metrics on a closed surface $S_g$ is a model for $E_{\mathcal{F}_{\text{fin}}}(\MCG(S_g))$, as discussed further in \cite[\S 4.10]{luck2005}. Thus, the identification of $\Tflat$ as a classifying space enriches the understanding of orbifold mapping class groups, offering a bridge between affine geometric structures and the homotopy theory of group actions.

For a group $G$ and a subgroup $H \leq G$, we denote the normalizer and the centralizer of $H$ in $G$ by $\nmlr[H]{G}$ and $\ctlr[H]{G}$, respectively.

\begin{thmx} \label{thm:classifying}
Let $\Gamma \leq \Aff(n)$ be a crystallographic group with holonomy $\hol = \tau(\Gamma)$, and let $\mathcal{O} = \Gamma\backslash\RR^n$ be the corresponding closed flat orbifold. Then the following hold:
\begin{enumerate}
\item \label{it:class-cocomens} There is an isomorphism 
\[
\MCG(\mathcal{O})/G \cong \tau(\nmlr[\Gamma]{\Aff(n)})/\hol
\]
for some finite normal subgroup $G \unlhd \MCG(\mathcal{O})$, and this isomorphism respects the actions of $\MCG(\mathcal{O})$ and of $\tau(\nmlr[\Gamma]{\Aff(n)})$ on $\Tflat$.
\item \label{it:class-EG} The space $\Tflat$ of flat metrics on $\mathcal{O}$ is model for a classifying space for proper actions of $\tau(\nmlr[\Gamma]{\Aff(n)})$ and of $\MCG(\mathcal{O})$. 
\end{enumerate}
\end{thmx}

\begin{rmk}
    It follows by inspection of the proof that the group $G$ in Theorem~\ref{thm:classifying} can be identified with the quotient $\nmlr[\Gamma]{\Aff(n)} \cap \tau^{-1}(\hol) \big/ \Gamma\ctlr[\Gamma]{\Aff(n)}$: see Lemma~\ref{lem:MCG-quotients}.  This allows us to compute the mapping class group $\MCG(\mathcal{O})$ of a closed flat orbifold $\mathcal{O}$ in the affine category.  We note that in other categories, such as $\mathbf{Top}$, $\mathbf{PL}$ or $\mathbf{Diff}$, the mapping class group might be different: see Remark~\ref{rmk:depends-on-category}.
\end{rmk}

Let $\Gamma \leq \Aff(n)$ be a crystallographic group with holonomy $\hol$ and translation lattice $\lat$. Then $\hol$ is a finite group acting on $\lat \leq \RR^n$ by linear transformations, and therefore $\lat \otimes \QQ$ is a representation of $\hol$. We then have splittings into irreducible $\hol$-representations
\begin{equation} \label{eq:splitting} \tag{$\ast$}
\lat \otimes \QQ = \bigoplus_{i=1}^\ell \bigoplus_{j=1}^{m_i} V_{i,j} \qquad\text{and}\qquad V_{i,j} \otimes K_i = \bigoplus_{\sigma \in \Gal(K_i/\QQ)} U_{i,j}^\sigma
\end{equation}
for some number fields $K_1,\ldots,K_\ell \subset \RR$ (which we choose to be minimal), so that $V_{i,j} \cong V_{i',j'}$ (as $\hol$-representations) if and only if $i = i'$, and so that $U_{i,j} \otimes_{K_i} \RR$ is an irreducible $\hol$-representation (over $\RR$) of type $\KK_i \in \{ \RR,\CC,\HH \}$.

The next result describes the group $\tau(\nmlr[\Gamma]{\Aff(n)})$ up to commensurability, i.e.\ taking finite-index subgroups and finite extensions.  In the case of fundamental groups of orbifolds, commensurability classification has the following geometric interpretation: if $G = \pi_1(\mathcal{O})$ is commensurable with a group $G'$, then $G' \cong \pi_1(\mathcal{O}')$ for an orbifold $\mathcal{O}'$ that is commensurable, i.e.\ shares a finite orbifold cover, with $\mathcal{O}$.

Because of certain technical complications, we assume that all the irreducible components of the $\hol$-representation $\lat \otimes \CC$ have Schur index one over $\CC$.  This is equivalent to saying that any irreducible $\hol$-subrepresentation $U' \leq \lat \otimes \CC$ is realizable over the field $K' = \QQ(\chi')$, where $\chi'\colon \hol \to \CC$ is the character of $U'$.  We note that for many classical groups, the Schur indices of \emph{all} of their irreducible representations are either one or two: see \cite[\S 1]{Turull} and references therein.

We note that, even though the description of $\tau(\nmlr[\Gamma]{\Aff(n)})$ up to isomorphism would be more desirable, we restrict ourselves to a description up to commensurability because of certain complications.  For instance, the action of $\tau(\nmlr[\Gamma]{\Aff(n)})$ does not necessarily preserve the subspaces $\bigoplus_{j=1}^{m_i} V_{i,j}$ appearing above---see e.g.\ the remark 3.2 for the manifold $G_6$ in \cite{Karla4D}--- and such a lack of splittings prevents certain algebraic descriptions of these groups.

\begin{thmx} \label{thm:commensurable}
With the notation above, suppose that all the irreducible constituents of the $\hol$-representation $\Lambda_\Gamma \otimes \CC$ have Schur index one over $\CC$.  Then there exists an isomorphism $\ctlr{\GL{\RR}} \cong \bigoplus_{i=1}^\ell \GL[m_i]{\KK_i}^{[K_i:\QQ]}$, under which the subgroup $\tau(\nmlr[\Gamma]{\Aff(n)}) \leq \GL{\RR}$ is commensurable with the subgroup
\[
\bigoplus_{i=1}^\ell \left\{ (A^\sigma \mid \sigma \in \Gal(K_i/\QQ)) \:\middle|\:\vphantom{\big|} A \in \GL[m_i]{\mathfrak{o}_i} \right\} \leq \bigoplus_{i=1}^\ell \GL[m_i]{\KK_i}^{[K_i:\QQ]},
\]
where $\mathfrak{o}_i$ is an $\mathfrak{o}_{K_i}$-order in a division $K_i$-subalgebra $L_i \subseteq \KK_i$ satisfying the following properties:
\begin{itemize}
    \item the canonical map $L_i \otimes_{K_i} \RR \to \KK_i$ is an isomorphism,
    \item the $\Gal(K_i/\QQ)$-action on $K_i$ extends to an action on $L_i$ by ring automorphisms, and
    \item the $\Gal(K_i/\QQ)$-action on $L_i$ restricts to an action on $\mathfrak{o}_i$.
\end{itemize}
\end{thmx}

The following result is a straightforward corollary of Theorem~\ref{thm:commensurable} and the fact that the moduli space of flat metrics, $\Mflat = \Tflat / \MCG(\mathcal{O})$, has the structure of a very good orbifold.

\begin{corx} \label{cor:commensurable-orb}
Let $\Gamma \leq \Aff(n)$ be a crystallographic group, and let $\mathcal{O} = \Gamma\backslash\RR^n$ be the corresponding closed flat orbifold. Then the moduli space for flat metrics on $\mathcal{O}$ is a very good orbifold commensurable with the Cartesian product
\[
\prod_{i=1}^\ell \left( \OO[m_i]{\KK_i}^{[K_i:\QQ]} \,\Big\backslash\, \GL[m_i]{\KK_i}^{[K_i:\QQ]} \,\Big/\, \left\{ (A^\sigma \mid \sigma \in \Gal(K_i/\QQ)) \,\middle|\,\vphantom{\big|} A \in \GL[m_i]{\mathfrak{o}_i} \right\} \right)
\]
in the notation and under the assumptions of Theorem~\ref{thm:commensurable}.
\end{corx}

Finally, we consider the case when all the $\hol$-representations $V_{i,j}$ in \eqref{eq:splitting} are irreducible over~$\RR$. In that case, we have the following easier description.

\begin{corx} \label{cor:commensurable-cryst}
Suppose that the subfields $K_i$ in~\eqref{eq:splitting} are all equal to $\QQ$, i.e.\ the $\hol$-representations $V_{i,j} \otimes \RR$ are irreducible over $\RR$. Then we have an identification $\ctlr{\GL{\RR}} \cong \bigoplus_{i=1}^\ell \GL[m_i]{\KK_i}$, under which the subgroup $\tau(\nmlr[\Gamma]{\Aff(n)}) \leq \GL{\RR}$ is commensurable with the subgroup
\[
\bigoplus_{i=1}^\ell \GL[m_i]{R_i} \leq \bigoplus_{i=1}^\ell \GL[m_i]{\KK_i},
\]
where $R_i \subseteq \KK_i$ is a discrete subring that spans $\KK_i$ over $\RR$.  In particular, $\Mflat$ is a very good orbifold commensurable with the Cartesian product
\[
\prod_{i=1}^\ell \left( \OO[m_i]{\KK_i} \,\big\backslash\, \GL[m_i]{\KK_i} \,\big/\, \GL[m_i]{R_i} \right).
\]
\end{corx}

The structure of the paper is as follows.  In Section~\ref{sec:prelims}, we review the necessary preliminaries on orbifolds and their Teichm\"uller and moduli spaces.  In Section~\ref{sec:EG}, we discuss the structure of the group $\MCG(\mathcal{O})$ and prove Theorem~\ref{thm:classifying}.  We study the group $\tau(\nmlr[\Gamma]{\Aff(n)})$ up to commensurability and prove the remaining results in Section~\ref{sec:commens}.

\subsection*{Acknowledgement}

K.\ G. and I.\ A.\ M.\ S. would like to thank the London Mathematical Society for partially supporting the project through the Scheme 5 Grant No.\ 52204.

\section{Preliminaries} \label{sec:prelims}

In this section we give some essential definitions related to orbifolds and their moduli spaces. See \cite{thurston1979} and \cite{gittins2024} together with the references therein for more details on orbifolds.
\subsection{Flat orbifolds} For a connected open subset $U\subseteq X$, an orbifold chart (of dimension~$n$) over $U$ is a triple $(\widetilde U, G_U,\phi_U)$, where $\widetilde U\subseteq\RR^n$ is a connected open subset, $G_U$ is a finite group acting effectively by diffeomorphisms on $\widetilde U$, and $\phi_U$ is a map inducing a homeomorphism ${\widetilde U/G_U\cong U}$.

A (smooth) \emph{orbifold} of dimension $n$ is a second countable Hausdorff topological space $X$ together with a maximal set of $n$-dimensional charts, called \emph{orbifold atlas}, that cover it. The orbifold charts must satisfy a compatibility condition (see, for example, Definition~2.1 of \cite{gittins2024}).

If an orbifold arises as the quotient space $M/G$ of a manifold $M$ under a smooth proper action of a discrete group $G$, it is called \emph{good}, and those which are quotients by finite groups are called \emph{very good}. If the action of the group is free, the quotient $M/G$ is a manifold.

A \emph{Riemannian structure} on an orbifold $\mathcal O$ is an assignment of a $G_U$-invariant Riemannian metric $\widetilde U$ to each orbifold chart. If $\mathcal{O}$ is a good orbifold, i.e.\ $\mathcal{O} = M/G$ for a smooth proper action of a discrete group $G$ on a manifold $M$, then the Riemannian structures on $\mathcal{O}$ can be naturally identified with $G$-invariant Riemannian structures on $M$.

A \emph{covering} of an orbifold $\mathcal O$ is a pair $(\widetilde {\mathcal O},\rho)$, where $\widetilde{\mathcal O}$ is another orbifold and $\rho\colon \widetilde X\to X $ is a smooth surjective map between the underlying topological spaces, satisfying the following property: each point $x\in X$ has a neighborhood $U$ and an orbifold chart over $U$, $(\widetilde U,G_U,\phi_U)$, for which each component $V_i\subseteq \rho^{-1}(U)$ is isomorphic to $\widetilde U/G_U^i$, where $G_U^i\subset G_U$ is a subgroup.

We say that a covering $(\widetilde{\mathcal{O}},\rho)$ of an orbifold $\mathcal{O}$ is \emph{finite} if the cardinality of the preimage of any point $x \in X$ under $\rho$ is finite. In particular, an orbifold is very good if and only if it has a finite covering $(\widetilde{\mathcal{O}},\rho)$ with $\widetilde{\mathcal{O}}$ a manifold.  Two orbifolds $\mathcal{O}$, $\mathcal{O}'$ are said to be \emph{commensurable} if they share a finite orbifold cover $\widetilde{\mathcal{O}}$.

The \emph{universal orbifold covering} of an orbifold $\mathcal O$ is a connected orbifold covering that is the covering of any other connected covering of $\mathcal O$. When $\mathcal O$ is a good orbifold, its universal orbifold covering is a simply connected manifold.  Let $(\widetilde{\mathcal O},\rho)$ be an orbifold covering. A \emph{deck transformation} of an orbifold covering is an (orbifold) homeomorphism $f:\widetilde {\mathcal O}\to \widetilde {\mathcal O}$ that preserves the covering map, that is, $\rho\circ f=\rho$. The \emph{orbifold fundamental group} of $\mathcal O$, $\pi^{orb}_1(\mathcal O)$, is the group of deck transformations of its universal cover.

\begin{definition}
A \emph{flat orbifold} is an orbifold equipped with a Riemannian structure in which the metric on each orbifold chart has zero sectional curvature. Such metrics are referred to as \emph{flat metrics}.
\end{definition}

Flat orbifolds are characterized by their orbifold fundamental group, which corresponds to a crystallographic group.

\begin{definition}
A \emph{crystallographic group} is a discrete subgroup $\Gamma$ of the group of isometries of $\RR^n$ such that $\RR^n/\Gamma$ is compact. A torsion-free crystallographic group $\Gamma$ is called a \emph{Bieberbach group}.
\end{definition}

When $\Gamma$ is a Bieberbach group we have that $\RR^n / \Gamma$ is a closed flat manifold. When $\Gamma$ is a non-torsion-free crystallographic group, we have that $\RR^n / \Gamma$ is a closed flat orbifold with singularities. Conversely, let $M$ be a closed manifold with a flat metric, then its universal cover is isometric to $\RR^n$ and its fundamental group is isomorphic to a Bieberbach group. Similarly, if a closed Riemannian orbifold is flat, then it is good (see \cite[Theorem III.$\mathcal G$.1.13]{bridson2013}), its universal orbifold covering is isometric to $\RR^n$ (by the Killing--Hopf theorem) and its orbifold fundamental group is isomorphic to a crystallographic group.

\subsection{Moduli space of flat metrics} \label{ssec:prelims-modsp}

Let $M$ be a smooth, closed manifold. We denote by $\mathcal{R}(M)$ the space of all (complete) Riemannian metrics on $M$. We equip $\mathcal{R}(M)$ with the smooth compact-open topology (see \cite{Tuschmann-Wraith}, \cite{Corro-Kordass}, \cite{Hirsch}). To identify metrics that are isometric, we consider the following group action: let $\Diff(M)$ denote the group of self-diffeomorphisms of $M$, then $\Diff(M)$ acts on $\mathcal{R}(M)$ by pulling back metrics. The quotient of $\mathcal{R}(M)$ by this action is the \emph{moduli space} $\mathcal{M}(M)$ of Riemannian metrics on $M$. Since this action preserves sectional curvature, we may consider the following:

\begin{definition} \label{def moduli}
The \emph{moduli space of flat metrics} $\Mflat[M]$ is the quotient of the space of (complete) Riemannian metrics with zero sectional curvature, $\Rflat[M]$, by the action of $\Diff(M)$.  The \emph{Teichm\"uller space of flat metrics} $\Tflat[M]$ is the quotient of $\Rflat[M]$ by the action of the subgroup $\Diff_0(M) \leq \Diff(M)$ of diffeomorphisms isotopic to the identity.
\end{definition}

Our methods to study the moduli spaces of flat metrics on orbifolds rely on a result of Bieberbach.

The group of affine transformations of $\RR^n$, denoted by $\Aff(n)$, has the structure of a semidirect product $\Aff(n) = \GL{\RR} \ltimes \RR^n $. The group of isometries of $\RR^n$, denoted by $\Iso(n)$, also has the structure of a semidirect product, $\Iso(n) = \OO{\RR}\ltimes \RR^n $. 

Consider the projection homomorphism
\begin{align*}
\tau\colon \Aff(n) & \rightarrow \GL{\RR}, \\
(A,v) & \mapsto A.  
\end{align*}

\begin{definition}
Let $\Gamma$ be a crystallographic group. The \emph{holonomy} of $\Gamma$, denoted  $\hol$,  is the subgroup of $\GL{\RR}$ given as the image of $\gp$ under $\tau$. The \emph{translation lattice} of $\Gamma$, denoted $\lat$, is the subgroup of $\Aff(n)$ given as the kernel of the restriction of $\tau$ to $\Gamma$.
\end{definition}

\begin{thm}[Bieberbach Theorems {\cite{Szczepanski}}] \label{thm:bieb}
Let $n \geq 0$.
\begin{enumerate}
\item Let $\Gamma \leq \Aff(n)$ be a crystallographic group. Then $\lat \cong \ZZ^n$ is a normal subgroup of $\Gamma$ of finite index, and it is a maximal abelian subgroup of $\Gamma$.
\item There exist finitely many isomorphism classes of crystallographic groups of dimension $n$.
\item \label{it:bieb-iso} Let $\phi\colon \Gamma \to \Gamma'$ be an isomorphism between two crystallographic groups $\Gamma,\Gamma' \leq \Aff(n)$. Then there exists $\alpha_\phi \in \Aff(n)$ such that $\phi(\gamma) = \alpha_\phi \circ \gamma \circ \alpha_\phi^{-1}$ for all $\gamma \in \Gamma$.
\end{enumerate}
\end{thm}

Thanks to the Bieberbach theorems, the space of flat metrics on a fixed closed orbifold $\mathcal{O}$ can be described in terms of their crystallographic groups, as follows. For a flat metric $g$ on $\mathcal{O}$ we get a crystallographic group $\Gamma$ such that $\mathcal{O}=\RR^n/\Gamma$, and if we consider another flat metric $g'$ on $\mathcal{O}$ we get another crystallographic group $\Gamma'$, where $\Gamma' = \gamma \Gamma \gamma^{-1} $ for some $\gamma \in \Aff(n)$.  In fact, if $\mathcal{O}$ is a manifold then we have the following description.

\begin{cor}[see \cite{Wolf2}, {\cite[\S 1.3]{KarlaThesis}}]
    For a flat Riemannian manifold $M = \RR^n/\gp$, where $\gp \leq \Aff(n)$ is a Bieberbach group, we have homeomorphisms
    \[
    \Tflat[M] \cong \Iso(n) \backslash \widetilde{\mathcal{C}}_\Gamma
    \]
    and
    \[
    \Mflat[M] \cong \Iso(n) \backslash \widetilde{\mathcal{C}}_\Gamma /  \nmlr[\gp]{\Aff(n)},
    \]
    where $\widetilde{\mathcal{C}}_\Gamma = \{ \gamma \in \Aff(n)  \mid \gamma \gp \gamma^{-1} \subset  \Iso(n) \}$.
\end{cor}

Since the subset $\widetilde{\mathcal{C}}_\Gamma$ above is the preimage of $\mathcal{C}_\Gamma \coloneqq \{ X \in \GL[n]{\RR} \mid X\hol X^{-1} \subset \OO[n]{\RR} \}$ under~$\tau$, the Teichm\"uller space and the moduli space of flat metrics of $M=\RR^n/\gp$ can be equivalently described as
\[
\Tflat[M] = \OO[n]{\RR} \backslash \mathcal{C}_\Gamma
\]
and
\[
\Mflat[M] = \OO[n]{\RR} \backslash \mathcal{C}_\Gamma / \tau ( \text{N}_{\text{Aff}(n)} (\gp)),
\]
respectively. This motivates generalisations of moduli and Teichm\"uller spaces in the case when $M$ is replaced by an orbifold, as follows.

\begin{definition}
Let $\mathcal{O}=\RR^n/\Gamma$ be a flat orbifold.
\begin{enumerate}
    \item The \emph{Teichm\"uller space of flat metrics} $\Tflat[\mathcal{O}]$ of $\mathcal{O}$ is the orbit space \[\OO[n]{\RR} \backslash \mathcal{C}_\Gamma,\] where $\mathcal{C}_\Gamma = \{ X \in \GL[n]{\RR} \mid X\hol X^{-1} \subset \OO[n]{\RR} \}$, under the left $\OO[n]{\RR}$-action by matrix multiplication.
    \item The \emph{moduli space of flat metrics} $\Mflat[\mathcal{O}]$ of $\mathcal{O}$ is the orbit space $\Tflat[\mathcal{O}] / \mathcal{N}_\Gamma$ under the right $\mathcal{N}_\Gamma$-action by matrix multiplication, where $\mathcal{N}_\Gamma = \tau(\nmlr[\gp]{\Aff(n)})$. 
\end{enumerate}
\end{definition}

Now, we explain the description of the Teichm\"uller space of flat metrics on a flat orbifold done by Bettiol, Derdzinski and Piccione \cite{Bettiol-Derdzinski-Piccione}, where the isotypic components are used. Recall that a non-zero invariant subspace $V$ of $\RR^n$ is \emph{irreducible} if it does not contain any proper invariant subspace, or, equivalently, if every non-zero element of the vector space $\End_{\hol}(V)$ of linear equivariant endomorphisms of $V$ is an isomorphism. In this situation, $\End_{\hol}(V)$ is an associative real division algebra and hence isomorphic to one of $\RR$, $\CC$, or  $\HH$. The irreducible subspace $V$ is called of \emph{real}, \emph{complex}, or \emph{quaternionic} type, according to the isomorphism type of $\End_{\hol}(V)$.

We now consider the following simplification of the splitting \eqref{eq:splitting} of the $\hol$-representation given in the Introduction, by decomposing the representation over $\RR$ to begin with instead of considering the decomposition over $\QQ$ first. In particular, let 
\[
\RR^n = \bigoplus_{i=1}^\ell \bigoplus_{j=1}^{M_i} U_{i,j}'
\]
be a decomposition into irreducible $\hol$-subspaces, so that $U'_{i,j}$ is isomorphic to $U'_{i',j'}$ if and only if $i=i'$. Then the $W'_i = \bigoplus_{j=1}^{M_i} U'_{i,j}$ are the so-called \emph{isotypic components}. They are said to be of type $\RR$, $\CC$ or $\HH$ if the irreducible representations $U'_{i,j}$ are of the corresponding type.  

\begin{thm}[Bettiol--Derdzinski--Piccione {[\cite{Bettiol-Derdzinski-Piccione}]}] \label{thm:Bettiol-Derdzinski-Piccione}
Let $\mathcal{O}$ be a closed flat orbifold, and denote by $W'_i$, $1\leq i \leq \ell$, the isotypic components of the orthogonal representation of its holonomy group. Each $W'_i$ consists of $M_i$ copies of the same irreducible representation, and we write $\mathbb{K}_i$ for $\RR$, $\CC$ or $\HH$, according to this representation being of real, complex, or quaternionic type. Then the Teichm\"uller space $\Tflat$ is diffeomorphic to
\[
\Tflat \cong \prod_{i=1}^\ell \frac{\GL[M_i]{\mathbb{K}_i}}{\OO[M_i]{\mathbb{K}_i}},
\]
where $\GL[m]{\mathbb{K}}$ is the group of $\mathbb{K}$-linear automorphisms of $\mathbb{K}^m$ and $\OO[m]{\mathbb{K}}$ stands for $\mathrm{O}(m)$, $\mathrm{U}(m)$, or $\mathrm{Sp}(m)$, when $\mathbb{K}$ is, respectively, $\RR$, $\CC$, or $\HH$. In particular, $\Tflat$ is real analytic and diffeomorphic to $\RR^d$.  
\end{thm}

The description of $\Tflat$ given in Theorem~\ref{thm:Bettiol-Derdzinski-Piccione} is consistent with the actions of $\tNG$ discussed below.  In particular, whenever a subgroup of $\GL{\RR} = \tau(\Aff(n))$ is identified with a subgroup $G \leq \GL[m]{\KK}$ in Section~\ref{sec:commens}, the action of $\tNG < \GL{\RR}$ on $\Tflat$ under this identification coincides with the action of $G$ on the right coset space $\OO[m]{\KK} \backslash \GL[m]{\KK}$ by right multiplication.

A study of $\Mflat[M]$ for $3$-dimensional closed manifolds was carried out by Kang in \cite{Kang} and by García in \cite{Karla3D}, where García examined the topology of these spaces using the structure of their Bieberbach groups and the above description. For $4$-dimensional closed flat manifolds, García and Palmas provided an algebraic description and analyzed the topology of most of these spaces in \cite{Karla4D}.

\subsection{The mapping class group for orbifolds} \label{ssec:prelims-mcg}

Here we introduce the main class of groups studied in the paper. We start with a general categorical setup.

We say a category $\mathcal{C}$ is \emph{topologically representable} if it is equipped with a faithful functor $F_{\mathcal{C}} \colon \mathcal{C} \to \mathbf{Top}$, where $\mathbf{Top}$ is the category of topological spaces (and continuous maps between them). The functor $F_{\mathcal{C}}$ is usually implicit---that is, we usually consider categories in which each object is understood to have an underlying topological space. Given two objects $X$, $Y$ in a topologically representable category $\mathcal{C}$, the set $\Hom_{\mathcal{C}}(X,Y)$ of morphisms from $X$ to $Y$ can be equipped with the \emph{compact-open topology}, generated by the subsets
\[
V(K,U) \coloneqq \{ f \in \Hom_{\mathcal{C}}(X,Y) \mid F_{\mathcal{C}}(f)(K) \subseteq U \}
\]
for $K \subseteq F_{\mathcal{C}}(X)$ compact and $U \subseteq F_{\mathcal{C}}(Y)$ open.

Now consider the case when $F_{\mathcal{C}}(X)$ is a compact Hausdorff space. Then the subset $\Aut_{\mathcal{C}}(X) \subseteq \Hom_{\mathcal{C}}(X,X)$ of automorphisms of $X$ is a topological group under the compact-open topology \cite[Theorem~3]{Arens}. Recall that for any topological group~$G$, the connected component $G_0 \subseteq G$ of the identity is a closed normal subgroup of $G$.

\begin{definition}
Let $X$ be an object in a topologically representable category $\mathcal{C}$ such that $F_{\mathcal{C}}(X)$ is a compact Hausdorff space. The \emph{mapping class group} of $X$ in $\mathcal{C}$ is the group
\[
\MCG_{\mathcal{C}}(X) = \pi_0(\Aut_{\mathcal{C}}(X)) = \Aut_{\mathcal{C}}(X)/\Aut_{\mathcal{C},0}(X),
\]
where $\Aut_{\mathcal{C},0}(X)$ is the connected component of the identity in $\Aut_{\mathcal{C}}(X)$.
\end{definition}

We are interested in mapping class groups of closed flat Riemannian orbifolds $\mathcal{O}$. We would like such a group to have a canonical action on the space of flat metrics on $\mathcal{O}$, i.e.\ we would like the automorphisms of $\mathcal{O}$ in our category to preserve the set of flat metrics on $\mathcal{O}$. In order to accomplish this, we define the \emph{affine category} $\mathbf{Aff}$, as follows. Let the objects of $\mathbf{Aff}$ be quotients $\Gamma \backslash \RR^n$ for $n \geq 0$ and discrete cocompact subgroups $\Gamma \leq \Aff(n)$. Given two objects $\mathcal{O} = \Gamma \backslash \RR^n$ and $\mathcal{O}' = \Gamma' \backslash \RR^{n'}$, let the morphisms $f\colon \mathcal{O} \to \mathcal{O}'$ in $\mathbf{Aff}$ be the maps between (the underlying spaces of) $\mathcal{O}$ and $\mathcal{O}'$ such that $f(\Gamma \cdot x) \subseteq \Gamma' \cdot \widetilde{f}(x)$ for all $x \in \RR^n$, where $\widetilde{f}\colon \RR^n \to \RR^{n'}$ is an \emph{affine map}---that is, a map defined by $\widetilde{f}(x) = Ax+b$ for some $A \in \operatorname{Mat}_{n' \times n}(\RR)$ and $b \in \RR^{n'}$. It is easy to see that a composite of two such morphisms in $\mathbf{Aff}$ is again a morphism, and so $\mathbf{Aff}$ is indeed a category.

Let $\mathcal{O} = \Gamma \backslash \RR^n$ be a closed flat Riemannian orbifold. One may then verify that $\Aut_{\mathbf{Aff}}(\mathcal{O})$ consists of bijective maps $f\colon \mathcal{O} \to \mathcal{O}$ such that $f(\Gamma \cdot x) = \Gamma \cdot \widetilde{f}(x)$ for all $x \in \RR^n$, where $\widetilde{f} \in \Aff(n)$. We call such a map $f$ an \emph{affine diffeomorphism} of $\mathcal{O}$, and we write $\AffDiff(\mathcal{O})$ and $\AffDiff_0(\mathcal{O})$ for $\Aut_{\mathbf{Aff}}(\mathcal{O})$ and $\Aut_{\mathbf{Aff},0}(\mathcal{O})$, respectively. For the ease of notation, we also write $\MCG(\mathcal{O})$ for the mapping class group $\MCG_{\mathbf{Aff}}(\mathcal{O}) = \AffDiff(\mathcal{O}) / \AffDiff_0(\mathcal{O})$.

\subsection{Classifying spaces}

We give a brief introduction to classifying spaces for families of subgroups. We follow \cite{luck2005}.

Given a discrete group $G$, a \emph{family} $\mathcal F$ of subgroups of $G$ is a non-empty set of subgroups that is closed under conjugation and taking subgroups in the following sense: if $H\in\mathcal F$, $g\in G$ and $K\leq H$, then $K\in\mathcal F$ and $gHg^{-1}\in\mathcal F$. We are considering the case when $\mathcal F$ is the family of finite subgroups of $G$. 

Let $\mathcal F$ be a family of subgroups of $G$. A model $E_{\mathcal F}(G)$ for the classifying $G$-$CW$-complex for the family $\mathcal F$ of subgroups is a $G$-$CW$-complex $E_{\mathcal F}(G)$ which has the following properties:
\begin{enumerate}
\item All isotropy groups of $E_{\mathcal F}(G)$ belong to $\mathcal F$.
\item For any $G$-$CW$-complex $Y$, whose isotropy groups belong to $\mathcal F$, there is up to $G$-homotopy precisely one $G$-map $Y \rightarrow X$.
\end{enumerate}

\begin{thm}[\protect{\cite{luck2005}, Theorem 1.9}] \label{thm:model}
Let $\mathcal F$ be a family of subgroups of $G$. A $G$-CW-complex $X$ is a model for  $E_\mathcal F(G)$ if and only if all its isotropy groups belong to $\mathcal F$ and for each $H \in \mathcal F $ the $H$-fixed point set $X^{H}$ is weakly contractible. 
\end{thm}

\subsection{Some algebraic number and Galois theories}

Here we review the concepts appearing in the statement of Theorem~\ref{thm:commensurable}.

Recall that given two modules $M$ and $N$ over a commutative ring $R$, the \emph{tensor product} of $M$ and $N$ over $R$ is an $R$-module $M \otimes_R N$ together with an $R$-bilinear map $\iota\colon M \times N \to M \otimes_R N$ satisfying the following universal property: for any $R$-module $L$ and any $R$-bilinear map $f\colon M \times N \to L$, there exists a unique $R$-linear map $f'\colon M \otimes_R N \to L$ such that $f = f' \circ \iota$. The elements of $M \otimes_R N$ are $R$-linear combinations of elements of the form $m \otimes n := \iota(m,n)$ for $m \in M$ and $n \in N$. Given an $R$-linear map $f\colon M \to M'$ and an $R$-module $N$, we also have an $R$-linear map $f \otimes_R N \colon M \otimes_R N \to M' \otimes_R N$ defined by $(f \otimes_R N)(m \otimes n) = f(m) \otimes n$. For simplicity of notation, we write $f \otimes N$ and $M \otimes N$, where $M$ is an $R$-module and $f\colon M \to M'$ an $R$-linear map, for $f \otimes_{\QQ} N$ and $M \otimes_{\QQ} N$ (respectively) if $M$ is a $\QQ$-vector space, and for $f \otimes_{\ZZ} N$ and $M \otimes_{\ZZ} N$ (respectively) otherwise.

By a \emph{number field} we mean a finite field extension of $\QQ$, i.e.\ a subfield $K \subseteq \CC$ such that $[K:\QQ] := \dim_{\QQ}(K) < \infty$. We say a number field $K$ is \emph{Galois} if the group $\Gal(K/\QQ) = \Aut(K)$ has order $[K:\QQ]$, and \emph{abelian} if furthermore $\Gal(K/\QQ)$ is abelian. By the Galois correspondence, given a Galois number field $K$ there is a bijection $\{ \text{subgroups of } \Gal(K/\QQ) \} \leftrightarrow \{ \text{subfields of } K \}$, given by $H \mapsto \operatorname{Fix}(H) = \{ x \in K \mid \sigma(x)=x \text{ for all } x \in H \}$ and $\Aut_{K'}(K) = \{ \sigma \in \Gal(K/\QQ) \mid \sigma|_{K'} = \Id_{K'} \} \mapsfrom K'$. Under this bijection, a subgroup $H \leq \Gal(K/\QQ)$ is normal if and only if the corresponding subfield $K' \subseteq K$ is Galois; in particular, if $K$ is an abelian number field then all subfields of $K$ are Galois.

Given a number field $K$, its \emph{ring of integers} $\mathfrak{o}_K$ is the set of elements $x \in K$ that are roots of a monic polynomial with coefficients in $\ZZ$; it can be shown that $\mathfrak{o}_K$ is a subring of $K$. Moreover, an associative $K$-algebra $L$ is a \emph{division algebra} if it has a multiplicative identity $1 \in L$ and every non-zero element $x \in L$ is invertible, i.e.\ there exists $x^{-1} \in L$ such that $xx^{-1} = 1 = x^{-1}x$. If $K$ is a number field and $L$ is a $K$-algebra, then an \emph{$\mathfrak{o}_K$-lattice} in $L$ is a finitely generated $\mathfrak{o}_K$-submodule of $L$ that spans $L$ over $K$; an $\mathfrak{o}_K$-lattice in $L$ that is a subring of $L$ is called an \emph{$\mathfrak{o}_K$-order}.

\section{Models for classifying spaces} \label{sec:EG}

In this section we prove Theorem~\ref{thm:classifying}.  Throughout this section, $\Gamma \leq \Aff(n)$ is a crystallographic group with holonomy $\hol = \tau(\Gamma)$ and translation lattice $\lat = \ker(\tau) \cap \Gamma$, where $\tau\colon \Aff(n) \to \GL{\RR}$ is the canonical map.  Recall that for a group $G$ and a subgroup $H \leq G$, we denote the normalizer and the centralizer of $H$ in $G$ by $\nmlr[H]{G}$ and $\ctlr[H]{G}$, respectively.

The following is an auxiliary lemma that will be used in this section as well as in the following section.

\begin{lem} \label{lem:ctlr-trans}
We have $\ctlr[\Gamma]{\Aff(n)} \leq \ker(\tau)$ and $\tau(\nmlr[\Gamma]{\Aff(n)}) \leq \nmlr{\Aut(\lat)}$.
\end{lem}

\begin{proof}
Note that an element $(A,v) \in \ctlr[\Gamma]{\Aff(n)}$ commutes with every translation $(\Id,w) \in \Gamma$, implying that $Aw = w$. Since the translation lattice $\lat \leq \RR^n$ of $\Gamma$ is cocompact and in particular spans $\RR^n$ over $\RR$, it follows that $A = \Id$ and thus $\ctlr[\Gamma]{\Aff(n)} \leq \ker(\tau)$, as required.

On the other hand, for any $\alpha = (A,v) \in \nmlr[\Gamma]{\Aff(n)}$ and any translation $(\Id,w) \in \Gamma$, we have $(A,v)^{\pm 1} (\Id,w) (A,v)^{\mp 1} = (\Id,A^{\pm 1}w) \in \Gamma$, implying that $\tau(\alpha) = A \in \Aut(\lat)$. Moreover, for any $B \in \hol$ we have $\beta \coloneqq (B,u) \in \Gamma$ for some $u \in \RR^n$ and therefore $A^{\pm 1} B A^{\mp 1} = \tau(\alpha^{\pm 1}\beta\alpha^{\mp 1}) \in \tau(\Gamma) = \hol$, implying that $\tau(\alpha) = A$ normalizes $\hol$. Thus indeed $\tau(\nmlr[\Gamma]{\Aff(n)})$ is a subgroup of $\nmlr{\Aut(\lat)}$.
\end{proof}

Theorem~\ref{thm:classifying}\ref{it:class-cocomens} is a consequence of Lemmas \ref{lem:MCG-norm} and \ref{lem:MCG-quotients} below.

\begin{lem} \label{lem:MCG-norm}
The group $\AffDiff(\mathcal{O})$ can be identified with the quotient $\nmlr[\Gamma]{\Aff(n)}/\Gamma$. Moreover, under this identification we have $\AffDiff_0(\mathcal{O}) = \Gamma\ctlr[\Gamma]{\Aff(n)}/\Gamma$.
\end{lem}

\begin{proof}
Let $\Gamma \leq \Aff(n)$ be the corresponding crystallographic group, so that $\mathcal{O} = \Gamma \backslash \RR^n$. Let $F \in \nmlr[\Gamma]{\Aff(n)}$. Then for any $\gamma \in \Gamma$ there exists $\gamma' \in \Gamma$ such that $F \circ \gamma = \gamma' \circ F$, implying that $F(\Gamma \cdot x) \subseteq \Gamma \cdot F(x)$ for all $x \in \RR^n$, and therefore $F$ induces a map $f\colon \mathcal{O} \to \mathcal{O}$ given by $f(\Gamma \cdot x) = \Gamma \cdot F(x)$. Moreover, the map $f$ is surjective since $F$ is surjective, and injective since $F(\Gamma \cdot x) = \Gamma \cdot F(x)$ for all $x \in \RR^n$ (which follows since for any $\gamma' \in \Gamma$ there exists $\gamma \in \Gamma$ such that $F \circ \gamma = \gamma' \circ F$). Thus $f$ is an affine diffeomorphism of $\mathcal{O}$, giving a map $\Phi\colon \nmlr[\Gamma]{\Aff(n)} \to \AffDiff(\mathcal{O})$ defined by $\Phi(F) = f$. It is easy to see that $\Phi$ is a continuous group homomorphism. We claim that $\Phi$ is surjective and has kernel $\Gamma$.

In order to see that $\Phi$ is surjective, note first that if an affine transformation $\alpha \in \Aff(n)$ satisfies $\alpha(x) \in \Gamma \cdot x$ for all $x \in \RR^n$, then $\alpha \in \Gamma$. Indeed, in this case for every $x \in \RR^n$ there exists $\gamma_x \in \Gamma$ such that $\alpha(x) = \gamma_x(x)$. On the other hand, if we take $x \in \RR^n$ such that $\Gamma \cdot x \in \mathcal{O}$ is a regular point then there exists a connected open neighbourhood $U \subseteq \RR^n$ of $x$ such that all points in $\Gamma \cdot U \subseteq \mathcal{O}$ are regular. The facts that $U$ is connected and $\alpha$ is continuous then imply that $\gamma_y = \gamma_x$ for all $y \in U$, and therefore $\alpha|_U = \gamma_x|_U$. Since $U$ is open and $\alpha,\gamma_x \in \Aff(n)$, it follows that $\alpha = \gamma_x \in \Gamma$, as required.

Now let $f \in \AffDiff(\mathcal{O})$, and let $\widetilde{f} \in \Aff(n)$ be a lift of $f$---that is, an affine linear transformation such that $f(\Gamma \cdot x) = \Gamma \cdot \widetilde{f}(x)$ for all $x \in \RR^n$.  We claim that the map $\widetilde{f} \in \Aff(n)$ normalizes~$\Gamma$. Indeed, since $f(\Gamma \cdot x) = \Gamma \cdot \widetilde{f}(x)$, it follows that $\widetilde{f}(\Gamma \cdot x) \subseteq \Gamma \cdot \widetilde{f}(x)$ for all $x \in \RR^n$ and hence $\widetilde{f}\Gamma\widetilde{f}^{-1} \cdot y \in \Gamma \cdot y$ for all $y \in \RR^n$. Given $\gamma \in \Gamma$, we thus have $\widetilde{f}\gamma\widetilde{f}^{-1} \cdot y \in \Gamma \cdot y$ for all $y \in \RR^n$, and therefore $\widetilde{f}\gamma\widetilde{f}^{-1} \in \Gamma$; this shows that $\widetilde{f}\Gamma\widetilde{f}^{-1} \subseteq \Gamma$. Conversely, let $\gamma \in \Gamma$. Given an $x \in \RR^n$, we have $\widetilde{f}(y) = \gamma \cdot \widetilde{f}(x)$, where $y = \widetilde{f}^{-1}\gamma\widetilde{f} \cdot x$, implying that $f(\Gamma \cdot y) = \Gamma \cdot \widetilde{f}(y) = \Gamma \cdot \widetilde{f}(x) = f(\Gamma \cdot x)$; since $f$ is injective, it follows that $\Gamma \cdot y = \Gamma \cdot x$ and therefore $\widetilde{f}^{-1}\gamma\widetilde{f} \cdot x = y \in \Gamma \cdot x$. Since this holds for all $x \in \RR^n$, it follows that $\widetilde{f}^{-1}\gamma\widetilde{f} \in \Gamma$, implying that $\widetilde{f}\Gamma\widetilde{f}^{-1} = \Gamma$, as claimed. This shows that $\Phi$ is surjective.

In order to compute $\ker(\Phi)$, let $f \in \AffDiff(\mathcal{O})$. Since $f(\Gamma \cdot x) = \Gamma \cdot \widetilde{f}(x)$ for all $x \in \RR^n$, we have $f = \Id$ if and only if $\widetilde{f}(x) \in \Gamma \cdot x$ for all $x \in \RR^n$, which happens if and only if $\widetilde{f} \in \Gamma$. Therefore, given $F \in \nmlr[\Gamma]{\Aff(n)}$, we have $\Phi(F) = \Id$ if and only if $F \in \Gamma$. Thus $\ker(\Phi) = \Gamma$, as required.  In particular, $\Phi = \overline\Phi \circ q_\Gamma$ for a continuous group isomorphism $\overline\Phi\colon \nmlr[\Gamma]{\Aff(n)}/\Gamma \to \AffDiff(\mathcal{O})$, where $q_\Gamma\colon \nmlr[\Gamma]{\Aff(n)} \to \nmlr[\Gamma]{\Aff(n)}/\Gamma$ is the quotient map.

In order to show that $\AffDiff_0(\mathcal{O})$ is identified with $\Gamma\ctlr[\Gamma]{\Aff(n)}/\Gamma$ under the isomorphism $\overline\Phi$, we need to show that $\Phi(\ctlr[\Gamma]{\Aff(n)}) = \AffDiff_0(\mathcal{O})$.  Note first that a translation $\lambda = (\Id,v) \in \Aff(n)$ commutes with a transformation $\alpha = (A,w) \in \Aff(n)$ if and only if $Av = v$, implying that if $(\Id,v)$ centralizes $\Gamma$ then so does $(\Id,tv)$ for any $t \in [0,1]$. This shows that $\ctlr[\Gamma]{\Aff(n)} \cap \ker(\tau) \subseteq \pi_0(\nmlr[\Gamma]{\Aff(n)})$, implying by Lemma~\ref{lem:ctlr-trans} that $\ctlr[\Gamma]{\Aff(n)} \subseteq \pi_0(\nmlr[\Gamma]{\Aff(n)})$.  Since $\Phi$ is continuous, and so sends connected subsets to connected subsets, it follows that $\Phi(\ctlr[\Gamma]{\Aff(n)}) \subseteq \AffDiff_0(\mathcal{O})$. Conversely, let $f \in \AffDiff_0(\mathcal{O})$. Then there exists a path $\varphi\colon [0,1] \to \nmlr[\Gamma]{\Aff(n)}/\Gamma$ such that $\varphi(0) = \Gamma$ and $\varphi(1) = \widetilde{f}\Gamma$, where $\widetilde{f} \in \Phi^{-1}(f)$. Since $\nmlr[\Gamma]{\Aff(n)}$ is a Lie group and $\Gamma$ is its discrete subgroup, it follows that the quotient map $q_\Gamma\colon \nmlr[\Gamma]{\Aff(n)} \to \nmlr[\Gamma]{\Aff(n)}/\Gamma$ is a covering map, and therefore $\varphi$ lifts to a path $\widetilde\varphi\colon [0,1] \to \nmlr[\Gamma]{\Aff(n)}$ such that $\widetilde\varphi(0) = \Id$ and $\widetilde\varphi(1) = \gamma' \widetilde{f}$ for some $\gamma' \in \Gamma$. Given $\gamma \in \Gamma$, we thus have a path $\psi_\gamma\colon [0,1] \to \Gamma$ given by $\psi_\gamma(t) = \widetilde\varphi(t) \gamma \widetilde\varphi(t)^{-1}$; since $\Gamma$ is discrete, it follows that $\psi_\gamma$ is constant and therefore $(\gamma'\widetilde{f}) \gamma (\gamma'\widetilde{f})^{-1} = \psi_\gamma(1) = \psi_\gamma(0) = \gamma$. Thus $\lambda = \gamma'\widetilde{f} \in \Aff(n)$ centralizes $\Gamma$; this shows that $\AffDiff_0(\mathcal{O}) = \Phi(\ctlr[\Gamma]{\Aff(n)})$, as required.
\end{proof}

\begin{rmk} \label{rmk:depends-on-category}
Lemma~\ref{lem:MCG-norm} allows us to identify $\MCG(\mathcal{O})$ with the outer automorphism group $\Out(\Gamma)$ of $\Gamma$.  Indeed, we have a map $\nmlr[\Gamma]{\Aff(n)} \to \Aut(\Gamma)$ defined by $\alpha \mapsto [\gamma \mapsto \alpha\gamma\alpha^{-1}]$, which is surjective by Theorem~\ref{thm:bieb}\ref{it:bieb-iso} and has kernel $\ctlr[\Gamma]{\Aff(n)}$ by definition. Thus we have an isomorphism $\Psi\colon \nmlr[\Gamma]{\Aff(n)}/\ctlr[\Gamma]{\Aff(n)} \to \Aut(\Gamma)$. Clearly, the subgroup of inner automorphisms of $\Gamma$ is precisely $\Psi(\Gamma\ctlr[\Gamma]{\Aff(n)}/\ctlr[\Gamma]{\Aff(n)})$, implying that we have an isomorphism $\nmlr[\Gamma]{\Aff(n)}/\Gamma\ctlr[\Gamma]{\Aff(n)} \cong \Out(\Gamma)$. But the mapping class group $\MCG(\mathcal{O})$ can be identified with $\nmlr[\Gamma]{\Aff(n)}/\Gamma\ctlr[\Gamma]{\Aff(n)}$ using Lemma~\ref{lem:MCG-norm}.

It is worth noting that a mapping class group can be defined in categories other than $\mathbf{Aff}$, such as $\mathbf{Top}$ (topological spaces and continuous maps), $\mathbf{PL}$ (piecewise linear manifolds and piecewise linear maps) and $\mathbf{Diff}$ (differentiable manifolds and smooth maps).  In \cite[Remark~4.7]{Bettiol-Derdzinski-Piccione}, an isomorphism $\MCG_{\mathcal{C}}(M) \cong \Out(\pi_1(M))$ for $\mathcal{C} = \mathbf{Diff}$ and a closed flat manifold $M$ is described.  However, we note that this isomorphism does not hold in general: for instance, for any $n$-torus $\mathbb{T}^n = \RR^n/\ZZ^n$, where $n \geq 5$, the canonical map $\MCG_{\mathcal{C}}(\mathbb{T}^n) \to \Out(\pi_1(\mathbb{T}^n)) = \GL{\ZZ}$ is not injective \cite[Theorem~4.1]{Hatcher-torus}.  The same is true for $\mathcal{C} = \mathbf{Top}$ and $\mathcal{C} = \mathbf{PL}$.

Moreover, even though this does not fall into the setup for mapping class groups above, one may also consider ``mapping class groups'' in the homotopy category $\mathcal{C} = \mathbf{hTop}$ (topological spaces and homotopy classes of continuous maps), where instead of a faithful functor $F_{\mathcal{C}} \colon \mathcal{C} \to \mathbf{Top}$ we consider the full forgetful functor $\mathbf{Top} \to \mathbf{hTop}$ and define the remaining concepts similarly to the construction above.  If $\mathcal{O}$ is a closed flat manifold, then $\MCG_{\mathbf{hTop}}(\mathcal{O})$---the group of self-homotopy equivalences of $\mathcal{O}$ modulo homotopy---can be identified with $\Out(\pi_1(\mathcal{O}))$: this follows since the universal cover of $\mathcal{O}$ is contractible \cite[Proposition~1B.9]{Hatcher}.  Using the notion of orbi-maps and homotopies between them \cite{Yamasaki,Takeuchi}, this identification can be generalised to any closed flat orbifold $\mathcal{O}$.
\end{rmk}

In view of Lemma~\ref{lem:MCG-norm}, from now on we will identify $\AffDiff(\mathcal{O})$ and $\MCG(\mathcal{O})$ with $\nmlr[\Gamma]{\Aff(n)}/\Gamma$ and $\pi_0(\nmlr[\Gamma]{\Aff(n)}/\Gamma)$, respectively.

Note that by construction, the action of $\AffDiff(\mathcal{O})$ on $\Tflat$ is induced by the action of $\nmlr[\Gamma]{\Aff(n)}$ on $\Iso(n) \backslash \{ \alpha \in \Aff(n) \mid \alpha\Gamma\alpha^{-1} \subseteq \Iso(n) \}$ by right multiplication. By Lemmas \ref{lem:ctlr-trans} and \ref{lem:MCG-norm}, any element of $\AffDiff_0(\mathcal{O})$ is of the form $\Gamma\lambda$ for a translation $\lambda$ and therefore acts trivially on $\Tflat$, implying that the action of $\AffDiff(\mathcal{O})$ induces an action of $\MCG(\mathcal{O})$ on $\Tflat$.

The following result relates $\MCG(\mathcal{O})$ to the group $\tNG = \tau(\nmlr[\Gamma]{\Aff(n)})$ studied later. Here, both of these groups are seen as quotients of $\nmlr[\Gamma]{\Aff(n)}$.

\begin{lem} \label{lem:MCG-quotients}
There exists a finite normal subgroup $G \unlhd \MCG(\mathcal{O})$ such that $\MCG(\mathcal{O})/G = \tNG/\hol$.
\end{lem}

\begin{proof}
For brevity, we will write $N$ and $C$ for the groups $\nmlr[\Gamma]{\Aff(n)}$ and $\ctlr[\Gamma]{\Aff(n)}$, respectively.

Note that given $\alpha \in N$ we have $\tau(\alpha)\hol\tau(\alpha)^{-1} = \tau(\alpha\Gamma\alpha^{-1}) = \tau(\Gamma) = \hol$, implying that $\hol$ is normal in $\tau(N)$. Thus, by Lemma~\ref{lem:MCG-norm}, we have well-defined groups $\MCG(\mathcal{O}) = (N/\Gamma)/(\Gamma C/\Gamma) = N/\Gamma C$ and $\tNG/\hol = (N/\ker(\tau|_N))/\hol = N/(\tau|_N)^{-1}(\hol)$. Since $\Gamma C \leq \Gamma\ker(\tau) = \tau^{-1}(\hol)$ (by Lemma~\ref{lem:ctlr-trans}) and $\Gamma C \leq N$, we also have $\Gamma C \leq (\tau|_N)^{-1}(\hol)$. It is then enough to show that $\Gamma C$ has finite index in $(\tau|_N)^{-1}(\hol)$.

Consider the $\RR$-linear map $\Theta\colon \RR^n \to \bigoplus_{h \in \hol} \RR^n$ defined by $\Theta(x) = (h(x)-x \mid h \in \hol)$. It is easy to verify that given any translation $\lambda_x = (\Id,x) \in \Aff(n)$, we have $\Theta(x) = 0$ if and only if $\lambda_x \in C$, whereas $\Theta(x) \in \bigoplus_{h \in \hol} \lat$ if and only if $\lambda_x \in N$. Since $C \leq \ker(\tau)$ by Lemma~\ref{lem:ctlr-trans}, it follows that $\Theta$ induces a continuous injection $\Theta'\colon \ker(\tau|_N)/C \to \bigoplus_{h \in \hol} \lat$. Since $\bigoplus_{h \in \hol} \lat$ is discrete in $\bigoplus_{h \in \hol} \RR^n$, it follows that $\ker(\tau|_N)/C$ is discrete in $\ker(\tau)/C$. However, since $\lat$ spans $\RR^n$ over $\RR$, it follows that $\lat C/C$ spans $\ker(\tau)/C$ over $\RR$. Since $\Gamma \subseteq N$ and $C \subseteq \ker(\tau|_N)$, we have $\lat C \subseteq \ker(\tau|_N)$, implying that $\ker(\tau|_N)/C$ and $\lat C/C$ are both cocompact lattices in $\ker(\tau)/C$; therefore, $\lat C$ has finite index in $\ker(\tau|_N)$. Since $\lat C \leq \Gamma C$ and since $\ker(\tau|_N)$ has finite index in $(\tau|_N)^{-1}(\hol)$, it follows that $\Gamma C$ has finite index in $(\tau|_N)^{-1}(\hol)$, as required.
\end{proof}

We now prove Theorem~\ref{thm:classifying}\ref{it:class-EG}.  In order to do that, we first identify the space $\Tflat$ with the space of $\hol$-invariant inner products in Lemma~\ref{lem:Teich-IP}.  We then use this description to show that $\Tflat$ is a classifying space for proper actions of $\tNG$ in Proposition~\ref{prop:EG}.

\begin{lem} \label{lem:Teich-IP}
Under the identification of $\OO{\RR} \backslash \GL{\RR}$ with the space of inner products on $\RR^n$ via $\OO{\RR} A \mapsto \langle {-},{-} \rangle_A$, where $\langle x,y \rangle_A \coloneqq x^T A^T A y$ for $x,y \in \RR^n$, the Teichm\"uller space $\OO{\RR} \backslash \mathcal{C}_\Gamma$ of flat metrics corresponds to the subspace of $\hol$-invariant inner products. Moreover, the action of $\tNG$ on the Teichm\"uller space can be described by $\langle {-},{-} \rangle \cdot B = \langle B{-},B{-} \rangle$ for $\langle {-},{-} \rangle \in \Tflat$ and $B \in \tNG$.
\end{lem}

\begin{proof}
Given $A \in \GL[n]{\RR}$, we have
\begin{align*}
\OO{\RR}A \in \OO{\RR} \backslash \mathcal{C}_\Gamma \quad &\Longleftrightarrow \quad A B A^{-1} \in \OO{\RR} \text{ for all } B \in \hol \\
&\Longleftrightarrow \quad \OO{\RR} A B = \OO{\RR} A \text{ for all } B \in \hol \\
&\Longleftrightarrow \quad \langle {-},{-} \rangle_{AB} = \langle {-},{-} \rangle_A \text{ for all } B \in \hol \\
&\Longleftrightarrow \quad \langle B{-},B{-} \rangle_A = \langle {-},{-} \rangle_A \text{ for all } B \in \hol \\
&\Longleftrightarrow \quad \langle {-},{-} \rangle_A \text{ is $\hol$-invariant},
\end{align*}
as required.

Since the group $\tNG$ acts on the Teichm\"uller space by right multiplication and since we have $\langle B{-},B{-} \rangle_A = \langle {-},{-} \rangle_{AB}$ for all $A$ and $B$, the ``moreover'' part follows immediately.
\end{proof}

The following result is well-known, but we include its proof here for completeness.

\begin{lem} \label{lem:dq-orbifold}
Let $G$ be a topological group and let $H,K \leq G$ be closed subgroups such that $H$ is discrete and $K$ is compact. Then the right $H$-action on the right coset space $K \backslash G$ is properly discontinuous.
\end{lem}

\begin{proof}
Let $U \subseteq K \backslash G$ be a compact subset. If $h \in H$ is such that $Uh \cap U \neq \varnothing$, then there exist $Ku,Ku' \in U$ such that $Kuh = Ku'$ and therefore $h \in u^{-1}Ku' \subseteq V^{-1}V$, where $V \subseteq G$ is the preimage of $U$ under the quotient map $q\colon G \to K \backslash G$. Note that the map $q$ is perfect \cite[Theorem~1.5.7]{Arhangelskii-Tkachenko}, implying that $V$ (and hence also $V^{-1}V$) is compact: see \cite[\S26, Exercise~12]{Munkres}. As $H$ is discrete, this implies that there are only finitely many elements $h \in H$ such that $Uh \cap U \neq \varnothing$, as required.
\end{proof}

\begin{prop} \label{prop:EG}
The space $\Tflat$ is a model of the classifying space for proper actions of~$\tNG$.
\end{prop}

\begin{proof}
By Theorem~\ref{thm:model}, it is enough to show that the action of $\tNG$ on $\Tflat$ is proper and that for any finite subgroup $F < \tNG$, the fixed point set of $F$ is contractible. We will use the description of $\Tflat$ as the space of $\hol$-invariant inner products on $\RR^n$, given by Lemma~\ref{lem:Teich-IP}.

Note first that $\tNG$ is discrete in $\GL{\RR}$, as it is contained in a discrete subgroup $\Aut(\lat) \cong \GL{\ZZ}$. Moreover, $\Tflat$ is a closed subspace of $\OO{\RR} \backslash \GL{\RR}$, and the action of $\tNG$ on $\Tflat$ extends to an action on $\OO{\RR} \backslash \GL{\RR}$. Since $\OO{\RR}$ is compact, it follows by Lemma~\ref{lem:dq-orbifold} that the $\tNG$-action on $\OO{\RR} \backslash \GL{\RR}$ is proper, hence the action on the subspace $\Tflat$ is proper as well.

Now let $F \leq \tNG$ be a finite subgroup; we aim to show that the fixed point set $\Fix(F)$ of $F$ is contractible. We first claim that $\Fix(F)$ is non-empty. Indeed, pick an arbitrary $\langle{-},{-}\rangle \in \Tflat$, and define $\langle{-},{-}\rangle'$ by $\langle x,y \rangle' = \sum_{A \in F} \langle Ax,Ay \rangle$. Then $\langle{-},{-}\rangle'$ is clearly an inner product on $\RR^n$ which is $F$-invariant. Moreover, $\langle{-},{-}\rangle'$ is $\hol$-invariant: indeed, given $B \in \hol$ and $x,y \in \RR^n$ we have
\[
\langle Bx,By \rangle' = \sum_{A \in F} \langle ABx,ABy \rangle = \sum_{A \in F} \langle (ABA^{-1})Ax,(ABA^{-1})Ay \rangle = \sum_{A \in F} \langle Ax,Ay \rangle = \langle x,y \rangle'
\]
since $\langle{-},{-}\rangle$ is $\hol$-invariant and $ABA^{-1} \in \hol$ for all $A \in F \leq \tNG$. Thus $\langle{-},{-}\rangle' \in \Fix(F)$ and so $\Fix(F)$ is non-empty, as claimed.

Finally, it remains to show that $\Fix(F)$ is contractible. Pick any $\langle{-},{-}\rangle_0 \in \Fix(F)$, and given $\langle{-},{-}\rangle \in \Fix(F)$ and $t \in [0,1]$ define $H(\langle{-},{-}\rangle,t) = (1-t)\langle{-},{-}\rangle + t\langle{-},{-}\rangle_0$. It is clear from the definition that $H(\langle{-},{-}\rangle,t)$ is an inner product that is both $\hol$-invariant and $F$-invariant, and therefore $H(\langle{-},{-}\rangle,t) \in \Fix(F)$, for any $\langle{-},{-}\rangle \in \Fix(F)$ and $t \in [0,1]$. Thus $H$ is a deformation retraction of $\Fix(F)$ to the point $\langle{-},{-}\rangle_0 \in \Fix(F)$, showing that $\Fix(F)$ is contractible, as required.
\end{proof}

\begin{proof}[Proof of Theorem~\ref{thm:classifying}]
    Part~\ref{it:class-cocomens} follows from the construction and Lemma~\ref{lem:MCG-quotients}. Furthermore, it follows from part~\ref{it:class-cocomens} that any finite subgroup of $\MCG(\mathcal{O})$ has the same fixed point set in $\Tflat$ as a finite subgroup of $\tNG$, and vice versa, implying by Theorem~\ref{thm:model} that $\MCG(\mathcal{O})$ and $\tNG$ have the same classifying space for proper actions. Part~\ref{it:class-EG} then follows from Proposition~\ref{prop:EG}.
\end{proof}

\section{Commensurability description} \label{sec:commens}

In this section we study the group $\tNG = \tau(\nmlr[\Gamma]{\Aff(n)})$ up to commensurability.  In Section~\ref{ssec:commens-orb} we show that $\Mflat$ is a very good orbifold, motivating the description of such an orbifold (and its orbifold fundamental group) up to commensurability.  In Section~\ref{ssec:commens-sgps}, we show that $\tNG$ is commensurable with a product of centralizers that is easier to treat in further arguments.  We prove Corollary~\ref{cor:commensurable-cryst} in Section~\ref{ssec:commens-cryst}; we prove Theorem~\ref{thm:commensurable} and deduce Corollary~\ref{cor:commensurable-orb} from it in Sections~\ref{ssec:commens-ncryst-fields}--\ref{ssec:commens-ncryst-lat}.

\subsection{Moduli spaces as orbifolds} \label{ssec:commens-orb}

We begin by describing the structure of $\Mflat$ as an orbifold. This will motivate the study of its orbifold fundamental group up to commensurability later, as groups commensurable to (i.e.\ sharing a finite-index subgroup with) the orbifold fundamental group of $\Mflat$ correspond to orbifolds commensurable to $\Mflat$.

\begin{prop} \label{prop:Mflat-orb}
The moduli space $\Mflat$ is a very good orbifold.
\end{prop}

\begin{proof}
By definition, the moduli space can be described as a double quotient, $\Mflat = \OO{\RR} \backslash \mathcal{C}_\Gamma / \tNG$, where $\mathcal{C}_\Gamma = \{ X \in \GL{\RR} \mid X \hol X^{-1} \subset \OO{\RR} \}$. Moreover, given any element $X \in \mathcal{C}_\Gamma$, the map $G \to \mathcal{C}_\Gamma$ given by $A \mapsto XA$ induces an isomorphism $K \backslash G \cong \OO{\RR} \backslash \mathcal{C}_\Gamma$, where $G = \nmlr{\GL{\RR}}$ and $K = \OO{\RR} \cap \nmlr{\GL{\RR}}$ \cite[\S 4.4]{Bettiol-Derdzinski-Piccione}.  It follows that $\Mflat$ can be described as a double coset space $K \backslash G / \tNG$. Note that we have $\tNG \leq \nmlr{\Aut(\lat)} = \Aut(\lat) \cap G$ by Lemma~\ref{lem:ctlr-trans}, and $\Aut(\lat) \cap G$ discrete in $G$ since $\Aut(\lat)$ is discrete in $\GL{\RR}$; thus $\tNG$ is discrete. The subgroup $K \leq G$ is a closed subgroup of the compact group $\OO{\RR}$, and is therefore compact itself. Therefore, by Lemma~\ref{lem:dq-orbifold}, the right $\tNG$-action on $K \backslash G$ is properly discontinuous. Since $K \backslash G$ is a manifold, this makes $\Mflat$ into a good orbifold.

Finally, we claim that $\tNG$ is virtually torsion-free, i.e.\ has a torsion-free subgroup $H$ of finite index: it will then follow that the double quotient $K \backslash G / H$ is a manifold (since the $H$-action on $K \backslash G$ is properly discontinuous and free) that is a finite orbifold cover of $\Mflat$, implying that $\Mflat$ is a very good orbifold, as required. To show the existence of such a subgroup $H$, note that $\tNG$ is a subgroup of the finitely generated linear group $\Aut(\lat) \cong \GL{\ZZ}$, and so it is virtually torsion-free by Selberg's Lemma \cite{Alperin}.
\end{proof}

\subsection{Commensurated subgroups} \label{ssec:commens-sgps}

We first relate the group $\tNG = \tau(\nmlr[\Gamma]{\Aff(n)})$ and the group $\nmlr{\Aut(\lat)}$, as follows:

\begin{prop} \label{prop:normaliser-fi}
The group $\tNG$ is a finite-index subgroup of $\nmlr{\Aut(\lat)}$.
\end{prop}

\begin{proof}
By Lemma~\ref{lem:ctlr-trans}, $\tNG$ is a subgroup of $\nmlr{\Aut(\lat)}$.

Now let $\mathcal{G}$ be the set of $n$-dimensional crystallographic groups whose translation lattice is $\lat$ and whose holonomy (viewed as a subgroup of $\Aut(\lat)$) is $\hol$. Then any group $\Gamma' \in \mathcal{G}$ can be uniquely described as
\[
\Gamma' = \{ (h,t) \in \Aff(n) \mid h \in \hol, t \in s_{\Gamma'}(h) \}
\]
for some function $s_{\Gamma'} \colon \hol \to \RR^n/\lat$. It is easy to verify (using the fact that $\Gamma'$ is a group) that $s_{\Gamma'}(gh) = g \cdot s_{\Gamma'}(h) + s_{\Gamma'}(g)$ for all $g,h \in \hol$, and that any function $s_{\Gamma'}\colon \hol \to \RR^n/\lat$ satisfying this condition defines a group $\Gamma' \in \mathcal{G}$. Moreover, if $\Gamma'' = (\Id,v)\Gamma'(\Id,v)^{-1}$ then we have $s_{\Gamma'}(h)-s_{\Gamma''}(h) = h \cdot v - v$ for all $h \in \hol$. There thus exists a bijection between the $\ker(\tau)$-conjugacy classes of groups in $\mathcal{G}$ and the first cohomology group $H^1(\hol,\RR^n/\lat) = \operatorname{Der}(\hol,\RR^n/\lat) / \operatorname{PDer}(\hol,\RR^n/\lat)$, where we define the \emph{derivations}
\[
\operatorname{Der}(\hol,\RR^n/\lat) = \{ s\colon \hol \to \RR^n/\lat \mid s(gh) = g \cdot s(h) + s(g) \text{ for all } g,h \in \hol \}
\]
and the \emph{principal derivations}
\[
\operatorname{PDer}(\hol,\RR^n/\lat) = \{ s_v'\colon \hol \to \RR^n/\lat, h \mapsto h \cdot v - v \mid v \in \RR^n \}.
\]
See \cite[\S2.2]{Szczepanski} for further details.

We claim that the abelian group $H^1 := H^1(\hol,\RR^n/\lat)$ is finite. Indeed, by \cite[Chapter~IV, Proposition~5.3]{MacLane}, every element $S \in H^1$ satisfies $kS = 0$, where $k = |\hol|$. If $S = [s]$ for some $s \in \operatorname{Der}(\hol,\RR^n/\lat)$ it then follows that $ks = s_v'$ for some $v \in \RR^n$ and therefore, by replacing $s$ with $s-s_{v/k}'$, we may assume that $ks = 0$. It follows that every element of $H^1$ is represented by some function $s\colon \hol \to \frac{1}{k}\lat/\lat$; since $\hol$ and $\frac{1}{k}\lat/\lat$ are both finite, it follows that $H^1$ is finite as well, as claimed.

Now note that $\nmlr{\Aut(\lat)}$ acts on $\mathcal{G}$ by setting $g \cdot \Gamma' = (g,0) \Gamma' (g,0)^{-1}$ for $g \in \nmlr{\Aut(\lat)}$ and $\Gamma' \in \mathcal{G}$. Moreover, if $\Gamma'' = (\Id,v) \Gamma' (\Id,v)^{-1}$ for some $v \in \RR^n$ it follows that $g \cdot \Gamma'' = (\Id,g \cdot v)(g \cdot \Gamma')(\Id,g \cdot v)^{-1}$, implying that $\nmlr{\Aut(\lat)}$ acts on the set of $\ker(\tau)$-conjugacy classes of subgroups in $\mathcal{G}$, and therefore it acts on $H^1(\hol,\RR^n/\lat)$ by the correspondence above.

Let $G_0$ be the stabiliser of $[s_\Gamma]$ under this action. Since $H^1(\hol,\RR^n/\lat)$ is finite, it follows that $G_0$ has finite index in $\nmlr{\Aut(\lat)}$. However, for any $g \in G_0$, we have $s_\Gamma - s_{g \cdot \Gamma} = s_v'$ and hence $(g,0)\Gamma(g,0)^{-1} = g \cdot \Gamma = (\Id,v)\Gamma(\Id,v)^{-1}$ for some $v \in \RR^n$, implying that $(g,-v) \in \Aff(n)$ normalizes $\Gamma$ and therefore $g = \tau(g,-v) \in \tau(\nmlr[\Gamma]{\Aff(n)})$. Thus $G_0 \subseteq \tNG$, and hence since $G_0$ has finite index in $\nmlr{\Aut(\lat)}$ so does $\tNG$, as required.
\end{proof}

We will be making repeated use of the following result.

\begin{lem} \label{lem:Aut-lat-fi}
Let $W$ be a finite-dimensional $\QQ$-vector space and $\Lambda,\Lambda' \subseteq W$ two lattices. Then the subgroups $\Aut(\Lambda)$ and $\Aut(\Lambda')$ of $\Aut(W)$ share a common finite-index subgroup.
\end{lem}

\begin{proof}
Since $\Lambda \otimes \QQ = W$, it follows that for any element $w \in W$ there exists $k_w \in \ZZ_{\geq 1}$ such that $k_w w \in \Lambda$. Therefore, since $t \coloneqq \dim(W) < \infty$ (and so $\Lambda' \cong \ZZ^t$), there exists $k \in \ZZ_{\geq 1}$ such that $k \Lambda' \subseteq \Lambda$. Therefore, $\Lambda'' \coloneqq \Lambda \cap \Lambda'$ has index $\leq k^t < \infty$ in $\Lambda'$, and therefore is a lattice in $W$. By symmetry, it is then enough to show the Lemma with $\Lambda'$ replaced by $\Lambda''$; in particular, it is enough to show the result in the case when $\Lambda' \subseteq \Lambda$.

Let $d = [\Lambda : \Lambda']$, and let $H = \Aut(\Lambda) \cap \Aut(\Lambda')$. Note that $\Aut(\Lambda)$ acts on the set of index-$d$ subgroups of $\Lambda$, and $H$ is precisely the stabiliser of $\Lambda'$ under this action; but since $\Lambda$ is finitely generated, it has finitely many subgroups of index $d$, implying that $H$ has finite index in $\Aut(\Lambda)$. On the other hand, note that $d\Lambda \subseteq \Lambda'$ and $[\Lambda':d\Lambda] = \frac{[\Lambda':d\Lambda']}{[d\Lambda:d\Lambda']} = d^{t-1} < \infty$. Since $\Aut(\Lambda')$ acts on the (finite) set of index-$d^{t-1}$ subgroups of $\Lambda'$ and $H$ is the stabiliser of $d\Lambda$ under this action, it follows that $H$ has finite index in $\Aut(\Lambda')$ as well, as required.
\end{proof}

We can split the representation of $\hol$ on $\lat \otimes \QQ \cong \QQ^n$ into irreducibles:
\[
    \lat \otimes \QQ = W_1 \oplus \cdots \oplus W_\ell, \quad\text{where } W_i = V_{i,1} \oplus \cdots \oplus V_{i,m_i}
\]
and each $\hol$-representation $V_{i,j}$ is irreducible, with $V_{i,j} \cong V_{i',j'}$ if and only if $i=i'$. Note that any two subgroups of $\GL[k]{\QQ}$ that are conjugate in $\GL[k]{\RR}$ are also conjugate in $\GL[k]{\QQ}$, as finding a conjugating matrix is equivalent to solving a system of linear equations with rational coefficients (together with one inequality for the determinant). Consequently, since the representations $V_{i,j}$ and $V_{i',j'}$ are non-isomorphic over $\QQ$ for $i \neq i'$, the representations $V_{i,j} \otimes \RR$ and $V_{i',j'} \otimes \RR$ are non-isomorphic over $\RR$.

The key point of this construction is that $W_i \otimes \RR$ is a crystallographic $\hol$-subspace that is preserved by any matrix centralising $\hol$. As a consequence, we may consider each of the subspaces $W_i$ separately, as follows.

\begin{lem} \label{lem:centr-split}
The group $\prod_{i=1}^\ell \ctlr[\hol|_{W_i}]{\Aut(\lat \cap W_i)}$ contains $\ctlr{\Aut(\lat)}$ as a finite index subgroup.
\end{lem}

\begin{proof}
As $\ctlr{\Aut(\lat \otimes \QQ)}$ preserves each $W_i$, it follows from the Schur's Lemma that we have $\ctlr{\Aut(\lat \otimes \QQ)} = \prod_{i=1}^\ell \ctlr[\hol|_{W_i}]{\Aut(W_i)}$. In particular, every element of $\ctlr{\Aut(\lat)}$ fixes both $W_i$ and $\lat$, so it must fix their intersection; it follows that $\ctlr{\Aut(\lat)}$ is indeed a subgroup of $\prod_{i=1}^\ell \ctlr[\hol|_{W_i}]{\Aut(\lat \cap W_i)}$.

Now as $W_i \leq \lat \otimes \QQ$, for each element of $W_i$ some of its non-zero integer multiples belong to $\lat$, implying that $\lat \cap W_i$ is a lattice in $W_i$. By comparing the ranks and dimensions, it follows that $\lat' \coloneqq (\lat \cap W_1) \oplus \cdots \oplus (\lat \cap W_\ell)$ is a lattice in $\lat \otimes \RR$, and therefore, by Lemma~\ref{lem:Aut-lat-fi}, $\Aut(\lat)$ and $\Aut(\lat')$ share a common finite-index subgroup. As $\ctlr{\Aut(\lat \otimes \QQ)} \cap \Aut(\lat) = \ctlr{\Aut(\lat)}$ and $\ctlr{\Aut(\lat \otimes \QQ)} \cap \Aut(\lat') = \prod_{i=1}^\ell \ctlr[\hol|_{W_i}]{\Aut(\lat \cap W_i)}$, the result follows.
\end{proof}

\begin{cor} \label{cor:commens}
    The subgroups $\tNG$ and $\prod_{i=1}^\ell \ctlr[\hol|_{W_i}]{\Aut(\lat \cap W_i)}$ of $\GL{\RR}$ share a common finite-index subgroup.
\end{cor}

\begin{proof}
    By Proposition~\ref{prop:normaliser-fi} and Lemma~\ref{lem:centr-split}, we have inclusions of finite-index subgroups $\tNG \leq \nmlr{\Aut(\lat)}$ and $\ctlr{\Aut(\lat)} \leq \prod_{i=1}^\ell \ctlr[\hol|_{W_i}]{\Aut(\lat \cap W_i)}$. It is thus enough to show that $\ctlr{\Aut(\lat)}$ has finite index in $\nmlr{\Aut(\lat)}$. But indeed, $\ctlr{\Aut(\lat)}$ is the kernel of the map $\nmlr{\Aut(\lat)} \to \Aut(\hol)$ sending a matrix $A$ to conjugation by $A$.  Since $\hol$ is finite, so is $\Aut(\hol)$, implying that $\ctlr{\Aut(\lat)}$ has finite index in $\nmlr{\Aut(\lat)}$, as required.
\end{proof}

We now study $\ctlr[\hol|_{W_i}]{\Aut(\lat \cap W_i)}$. In order to simplify the notation, we write $W \coloneqq W_i$, $V_j \coloneqq V_{i,j}$ and $m \coloneqq m_i$. Slightly abusing the notation, we will also identify the subgroup $\hol \leq \GL{\RR}$ with its restriction $\hol|_U \leq \Aut(U) = \{ A \in \GL{\RR} \mid AU = U \}$ for an $\hol$-invariant subset $U \subseteq \RR^n$. Therefore, $W$ is a representation of $\hol$ over $\QQ$ with an invariant lattice $\lat \cap W \subset W$, and we have a splitting $W = V_1 \oplus \cdots \oplus V_m$, where the $V_j$ are pairwise isomorphic and irreducible over $\QQ$.

\subsection{The crystallographic case} \label{ssec:commens-cryst}

The first case is when we have an irreducible crystallographic subrepresentation: that is, when $V_j \otimes \RR$ is irreducible over $\RR$ for some (equivalently, any) $j$. Even though the result below follows from the more general considerations in the next two subsections, we have decided to give a shorter and more straightforward proof separately.

\begin{lem} \label{lem:crystallographic}
Suppose each $V_j \otimes \RR$ is an irreducible representation over $\RR$ of type $\KK$, where $\KK \in \{ \RR, \CC, \HH \}$. Then, under the isomorphism $\ctlr{\Aut(W \otimes \RR)} \cong \GL[m]{\KK}$, a finite index subgroup of $\ctlr{\Aut(\lat \cap W)}$ is identified with a finite-index subgroup of $\GL[m]{R}$, where $R \subseteq \KK$ is a discrete subring such that $(R,+)$ is a lattice in the $\RR$-vector space $(\KK,+)$.
\end{lem}

\begin{proof}
We first explicitly describe the isomorphism $\ctlr{\Aut(W \otimes \RR)} \cong \GL[m]{\KK}$. We have $\KK = \End_{\hol}(V \otimes \RR)$, where $V$ is an irreducible $\QQ$-representation of $\hol$ such that there exist $\hol$-equivariant isomorphisms $\phi_j\colon V \to V_j$ of $\QQ$-vector spaces, for $1 \leq j \leq m$. The isomorphism $\Phi\colon \ctlr{\Aut(W \otimes \RR)} \to \GL[m]{\KK}$ then sends an automorphism $\alpha$ of $W \otimes \RR$ to the matrix $A = (a_{pq})_{1 \leq p,q \leq m}$ such that $\alpha|_{V_q \otimes \RR} \circ (\phi_q \otimes \RR) = \sum_{p=1}^m (\phi_p \otimes \RR) \circ a_{pq}$ for every $q$.

Now let $\lat' = \bigcap_{j=1}^m \phi_j^{-1}(\lat \cap V_j)$. As $V_j \leq (\lat \cap W) \otimes \QQ$ for $2 \leq j \leq m$, for each element of $\lat \cap V_1$ some of its non-zero integer multiples are mapped by $\phi_j \circ \phi_1^{-1}$ to $\lat \cap V_j$, implying that $\phi_1(\lat')$ has finite index in $\lat \cap V_1$; in particular, $\lat'$ is a lattice in $V$ and so in $V \otimes \RR$. It follows that $V \otimes \RR = \lat' \otimes \RR$ and therefore $R \coloneqq \End_{\hol}(\lat')$ can be identified with a subring of~$\KK$, which is discrete since $\lat'$ is finitely generated and discrete in $V \otimes \RR$.

As $R$ is discrete in $\KK$, in order to show that $(R,+)$ is a lattice in $(\KK,+)$ it is enough to show that $(R,+)$ spans $(\KK,+)$ over $\RR$. But since $V$ is dense in $V \otimes \RR$, it follows that $\KK = \End_{\hol}(V \otimes \RR)$ is a subspace of $\End(V \otimes \RR)$ containing $\End_{\hol}(V)$ as a dense subset. Now as $\lat' < V$ is a lattice, some non-zero integer multiple of any given element of $R = \End_{\hol}(\lat')$ will belong to $\End_{\hol}(V)$, implying that $(R,+)$ spans $(\KK,+)$ over $\RR$, as required.

Now it is clear from the construction that $\Phi^{-1}(\GL[m]{R})$ is precisely the group of automorphisms of $W$ that commute with $\hol$ and leave the lattice \[ \overline{\lat} \coloneqq \phi_1(\lat') \oplus \cdots \oplus \phi_{m_i}(\lat') \] invariant. In particular, we have $\Phi^{-1}(\GL[m]{R}) = \ctlr{\Aut(W)} \cap \Aut(\overline{\lat})$; note also that $\ctlr{\Aut(\lat \cap W)} = \ctlr{\Aut(W)} \cap \Aut(\lat \cap W)$. The result then follows from Lemma~\ref{lem:Aut-lat-fi}.
\end{proof}

\begin{proof}[Proof of Corollary~\ref{cor:commensurable-cryst}]
    The description of $\tau(\nmlr[\Gamma]{\Aff(n)})$ follows directly from Lemma~\ref{lem:crystallographic} and Corollary~\ref{cor:commens}. The description of $\Mflat$ then follows from Proposition~\ref{prop:Mflat-orb} together with Theorem~\ref{thm:Bettiol-Derdzinski-Piccione} and the surrounding discussion.
\end{proof}

\subsection{The non-crystallographic case: fields} \label{ssec:commens-ncryst-fields}

In order to describe $\ctlr{\Aut(\lat \cap W)}$, we first describe $\ctlr{\Aut(W)}$.  We will discuss the restriction to $\Aut(\lat \cap W)$ in the next subsections.

\begin{sta} \label{sta:Schur1}
    From now on, we assume that the character of every (equivalently, some) irreducible constituent of $V_j \otimes \CC$ for every (equivalently, some) $j \in \{1,\ldots,m\}$ has Schur index one over $\QQ$. 
\end{sta}

Let $K$ be a Galois number field.  Note that $\Gal(K/\QQ)$-actions on $V_j$ (trivially) and $K$ (canonically) induce a $\Gal(K/\QQ)$-action on $V_j \otimes K$. Moreover, given any $h \in \hol$, $\sigma \in \Gal(K/\QQ)$, $v \in V_j$ and $k \in K$, we have
\[
h \cdot (v \otimes k)^\sigma = h \cdot (v \otimes k^\sigma) = (h \cdot v) \otimes k^\sigma = ((h \cdot v) \otimes k)^\sigma = (h \cdot (v \otimes k))^\sigma,
\]
implying that $h \cdot w^\sigma = (h \cdot w)^\sigma$ for all $w \in V_j \otimes K$. Note that the map $V_j \otimes K \to V_j \otimes K, w \mapsto w^\sigma$ is \emph{not} $K$-linear, but rather satisfies the equation $(kw)^\sigma = k^\sigma w^\sigma$ for $k \in K$ and $w \in V_j \otimes K$, and in particular maps $K$-linear subspaces to $K$-linear subspaces. Given a subset $U \subseteq V_j \otimes K$, we write $U^\sigma \coloneqq \{ u^\sigma \mid u \in U \}$.

In the next two results, we describe the $\hol$-representation $V_j$ in terms of an irreducible component of $V_j \otimes \RR$.

\begin{lem} \label{lem:Gal-premutes}
    If $U_j' \leq V_j \otimes \RR$ is an irreducible $\hol$-subspace with character $\chi\colon \hol \to \RR$, then $U_j' \cong U_j \otimes_{K} \RR$, where $U_j$ is an irreducible $K$-linear $\hol$-subspace for $K = \QQ(\chi)$.  Furthermore, we have $V_j \otimes K = \bigoplus_{\sigma \in \Gal(K/\QQ)} U_j^\sigma$.
\end{lem}

\begin{proof}
    Suppose first that $U_j' \otimes_\RR \CC$ is irreducible over $\CC$.  Since $U_j' \otimes_\RR \CC$ has Schur index one, it is realizable over $K$, and so $U_j' = U_j \otimes_{K} \RR$ for an irreducible $K$-linear $\hol$-subspace $U_j$.  Then, by \cite[Theorem~9.21]{Isaacs}, the irreducible components of $V_j \otimes K$, up to similarity, are precisely $U_j^\sigma$ for $\sigma \in \Gal(K/\QQ)$: that is, $V_j \otimes K = \bigoplus_{\sigma \in \Gal(K/\QQ)} U_{j,\sigma}$ for some irreducible $K$-linear $\hol$-subspaces $U_{j,\sigma} \leq V_j \otimes K$, with $U_{j,\sigma} \cong U_j^\sigma$.

    Now since $K = \QQ(\chi)$ we have $\chi^\sigma \neq \chi^\tau$, and therefore $U_j^\sigma \not\cong U_j^\tau$ for any distinct elements $\sigma,\tau \in \Gal(K/\QQ)$, implying that the subrepresentations $\{ U_{j,\sigma} \mid \sigma \in \Gal(K/\QQ) \}$ are pairwise non-isomorphic. In particular, given $\sigma \in \Gal(K/\QQ)$, it follows from Schur's Lemma that $U_{j,\sigma}$ is the unique $\hol$-subrepresentation of $V_j \otimes K$ that is isomorphic to $U_j^\sigma$; since $U_{j,1}^\sigma$ is another such subrepresentation, we have $U_{j,\sigma} = U_{j,1}^\sigma$.  It follows that $V_j \otimes K = \bigoplus_{\sigma \in \Gal(K/\QQ)} U_j^\sigma$ for $U_j = U_{j,1}$, as required.

    Suppose now instead that $U_j' \otimes_\RR \CC$ is not irreducible, and let $\widehat{U}_j'$ be its irreducible constituent. Then, since $\widehat{U}_j'$ has Schur index one, we have $\widehat{U}_j' = \widehat{U}_j \otimes_{\widehat{K}} \CC$ for an irreducible $\widehat{K}$-linear $\hol$-subspace $\widehat{U}_j$, where $\widehat{K} = \QQ(\widehat{\chi})$ and $\widehat{\chi}$ is the character of $\widehat{U}_j$. Moreover, by the argument above, we have $V_j \otimes \widehat{K} = \bigoplus_{\sigma \in \Gal(\widehat{K}/\QQ)} \widehat{U}_j^\sigma$.

    Now let $\iota \in \Gal(\widehat{K}/\QQ)$ be the complex conjugation, and let $U_j = \{ u + u^\iota \mid u \in \widehat{U}_j \}$. Then $U_j$ is a $K$-linear $\hol$-invariant subspace of $\widetilde{U}_j \coloneqq (V_j \otimes K) \cap (\widehat{U}_j \oplus \widehat{U}_j^\iota)$, where $K = \widehat{K} \cap \RR$; here the $\hol$-invariance follows from the fact that $\Gal(K/\QQ)$ is abelian, as $K$ is a subfield of a cyclotomic field.  Therefore, $U_j$ is either zero or equal to $\widetilde{U}_j$ (since $\widehat{U}_j$ is not realizable over $\RR$ while $\widetilde{U}_j$ is, implying that $\widetilde{U}_j$ is irreducible over $K$ as an $\hol$-subspace).  However, if $U_j$ was zero then this would imply that each $u \in \widehat{U}_j$ is a sum of elements of the form $v \otimes \mathbf{i}k$ for $v \in V_j$ and $k \in \RR$, implying that if $\kappa$ is any non-real element of $\widehat{K}$ then $(\kappa - \kappa^\iota)\widehat{U}_j \subseteq V_j \otimes K$, contradicting the fact that $\widehat{U}_j$ is not realizable over $\RR$. Therefore, we must have $U_j = \widetilde{U}_j$.

    It now follows, by considering dimensions, that $\bigoplus_{\sigma \in \Gal(K/\QQ)} U_j^\sigma$ is the whole of $V_j \otimes K$. Furthermore, we have $\QQ(\chi) \subseteq \widehat{K} \cap \RR = K$, whereas the fact that $\{ \widehat{U}_j^\sigma \mid \sigma \in \Gal(\widehat{K}/\QQ) \}$ are pairwise non-isomorphic implies that so are $\{ U_j^\sigma \mid \sigma \in \Gal(K/\QQ) \}$, and thus $\chi^\sigma \neq \chi$ for $1 \neq \sigma \in \Gal(K/\QQ)$. Hence $K = \QQ(\chi)$, as required.
\end{proof}

\begin{cor} \label{cor:Vj}
In the notation of Lemma~\ref{lem:Gal-premutes}, we have $V_j = \left\{ \sum_{\sigma \in \Gal(K/\QQ)} u^\sigma \,\middle|\, u \in U_j \right\}$.
\end{cor}

\begin{proof}
We will show that both $V_j$ and $V_j' \coloneqq \left\{ \sum_{\sigma \in \Gal(K/\QQ)} u^\sigma \,\middle|\, u \in U_j \right\}$ are equal to the fixed set of $\Gal(K/\QQ)$, i.e.\ equal to $X \coloneqq \{ v \in V_j \otimes K \mid v^\sigma = v \text{ for all } \sigma \in \Gal(K/\QQ) \}$.

It is clear that $V_j \subseteq X$. On the other hand, let $(e_1,\ldots,e_q)$ be a basis for $V_j$ over $\QQ$. Then any element $v \in V_j \otimes K$ has a unique expression of the form $v = \sum_{i=1}^q e_i \otimes k_i$ for some $k_1,\ldots,k_q \in K$. If $v \in X$, we then have $\sum_{i=1}^q e_i \otimes (k_i-k_i^\sigma) = v-v^\sigma = 0$ for all $\sigma \in \Gal(K/\QQ)$, implying (by the uniqueness of expressions) that $k_i^\sigma = k_i$ for all $\sigma \in \Gal(K/\QQ)$ and hence $k_i \in \QQ$. Thus $v \in V_j$, as required.

Now it is clear that $V_j' \subseteq X$. On the other hand, since $V_j \otimes K = \bigoplus_{\sigma \in \Gal(K/\QQ)} U_j^\sigma$ by Lemma~\ref{lem:Gal-premutes}, any element $v \in V_j \otimes K$ has a unique expression of the form $v = \sum_{\sigma \in \Gal(K/\QQ)} u_\sigma^\sigma$ for some elements $u_\sigma \in U_j$. If $v \in X$ and $\tau \in \Gal(K/\QQ)$, we then have $v = v^{\tau} = \sum_{\sigma \in \Gal(K/\QQ)} u_\sigma^{\sigma\tau}$ and therefore, by the uniqueness of the expression $v = \sum_{\sigma \in \Gal(K/\QQ)} u_\sigma^{\sigma}$, we have $u_{\sigma\tau}^{\sigma\tau} = u_{\sigma}^{\sigma\tau}$ for all $\sigma \in \Gal(K/\QQ)$ and in particular $u_\tau^\tau = u_{\Id}^\tau$. Thus $v = \sum_{\sigma \in \Gal(K/\QQ)} u_{\Id}^{\sigma} \in V_j'$, as required.
\end{proof}

It follows that we have
\[
W \otimes K = \bigoplus_{\sigma \in \Gal(K/\QQ)} T^\sigma, \qquad\text{where}\quad T = \bigoplus_{j=1}^m U_j,
\]
and where $W = \{ \sum_{\sigma \in \Gal(K/\QQ)} t^\sigma \mid t \in T \}$; here the $\Gal(K/\QQ)$ action on $W \otimes K$ is defined as the diagonal action on all the $V_j \otimes K$. Note that two irreducible $K$-linear $\hol$-subspaces $U \leq T^{\sigma}$ and $U' \leq T^{\sigma'}$ are isomorphic if and only if $\sigma = \sigma'$.

Given any $\alpha \in \Aut_K(T)$ and $\sigma \in \Gal(K/\QQ)$, we may define a map $\alpha^\sigma\colon T^\sigma \to T^\sigma$ by setting $\alpha^\sigma(t^\sigma) = \alpha(t)^\sigma$ for $t \in T$; it is easy to verify that $\alpha^\sigma$ is linear over $K$, i.e.\ $\alpha^\sigma \in \Aut_K(T^\sigma)$. We may also define a map $\widetilde\alpha \in \Aut_K(W \otimes K)$ by setting $\widetilde\alpha(\sum_{\sigma \in \Gal(K/\QQ)} t_\sigma) = \sum_{\sigma \in \Gal(K/\QQ)} \alpha^\sigma(t_\sigma)$ whenever $t_\sigma \in T^\sigma$.  This allows us to describe the centralizer of $\hol$ in $\Aut(W)$, as follows.

\begin{lem} \label{lem:ctlrK}
We have $\ctlr{\Aut(W)} = \{ \widetilde\alpha|_W \mid \alpha \in \ctlr{\Aut_K(T)} \}$.
\end{lem}

\begin{proof}
Note first that $\widetilde\alpha|_W \in \Aut(W)$ for all $\alpha \in \Aut_K(T)$. Indeed, by the description of $W \otimes K$ above, given $w \in W$ we have $w = \sum_{\sigma \in \Gal(K/\QQ)} t^\sigma$ for some $t \in T$, implying that $\widetilde\alpha(w) = \sum_{\sigma \in \Gal(K/\QQ)} \alpha^\sigma(t^\sigma) = \sum_{\sigma \in \Gal(K/\QQ)} \alpha(t)^\sigma \in W$, as required. Since $h \cdot w^\sigma = (h \cdot w)^\sigma$ for all $h \in \hol$, $\sigma \in \Gal(K/\QQ)$ and $w \in W \otimes K$, it also follows that $\widetilde\alpha|_W$ centralizes $\hol$ whenever $\alpha$ does.

Conversely, let $\beta \in \ctlr{\Aut(W)}$, and let $\beta' = \beta \otimes K \in \Aut_K(W \otimes K)$. Since $\beta'$ centralizes $\hol$, it follows from the decomposition of $W \otimes K$ and Schur's Lemma that there exist automorphisms $\beta_\sigma \in \ctlr{\Aut_K(T^\sigma)}$ such that $\beta'(\sum_{\sigma \in \Gal(K/\QQ)} t_\sigma) = \sum_{\sigma \in \Gal(K/\QQ)} \beta_\sigma(t_\sigma)$ whenever $t_\sigma \in T^\sigma$. It is now enough to show that $\beta_\sigma = \beta_{\Id}^\sigma$ for all $\sigma$. But note that since $\beta \in \Aut(W)$ we have $\beta(\sum_{\sigma \in \Gal(K/\QQ)} t^\sigma) = \sum_{\sigma \in \Gal(K/\QQ)} \beta_\sigma(t^\sigma) \in W$ for all $t \in T$, implying that $\beta_{\Id}^\sigma(t^\sigma) = \beta_{\Id}(t)^\sigma = \beta_\sigma(t^\sigma)$; since this holds for every $t \in T$, we indeed have $\beta_\sigma = \beta_{\Id}^\sigma$, as required.
\end{proof}

\subsection{Description of the division algebra} \label{ssec:commens-ncryst-divalg}

In order to describe the groups $\ctlr{\Aut(W)}$ and $\ctlr{\Aut(\lat \cap W)}$, we first give a description $\ctlr{\Aut_K(T)}$ as $\GL[m]{L}$ for some associative division algebra $L$.

\begin{prop} \label{prop:AutU-ncryst}
Suppose $U_1 \otimes_K \RR$ (and thus each $U_j \otimes_K \RR$) is an irreducible $\RR$-representation of type $\KK \in \{ \RR, \CC, \HH \}$. Then, under the isomorphism $\ctlr{\Aut_\RR(T \otimes_K \RR)} \cong \GL[m]{\KK}$, the subgroup $\ctlr{\Aut_K(T)} \leq \ctlr{\Aut_\RR(T \otimes_K \RR)}$ is identified with $\GL[m]{L}$, where $L \subseteq \KK$ is an associative division algebra over $K$ such that the action of every $\sigma \in \Gal(K/\QQ)$ on $K$ extends to a ring isomorphism $L \cong L^\sigma$ for some $K$-subalgebra $L^\sigma \subseteq \KK$, and such that the canonical map $L \otimes_K \RR \to \KK$ is an isomorphism.
\end{prop}

\begin{proof}
The isomorphism $\ctlr{\Aut_\RR(T \otimes_K \RR)} \cong \GL[m]{\KK}$ is induced, under some $\hol$-equivariant isomorphisms $U_1 \otimes_K \RR \to U_j \otimes_K \RR$, by the isomorphism $\End_{\hol}(U_1 \otimes_K \RR) \cong \KK$. Let $L \subseteq \KK$ correspond to $\End_{\hol}(U_1) \subseteq \End_{\hol}(U_1 \otimes_K \RR)$ under the latter isomorphism. Since, by construction, the $\hol$-equivariant isomorphisms $U_1 \otimes_K \RR \to U_j \otimes_K \RR$ are induced by $\hol$-equivariant isomorphisms $U_1 \to U_j$, it follows that $\ctlr{\Aut_K(T)} \cong \GL[m]{L}$.

Now it is clear that $L \cong \End_{\hol}(U_1)$ is stable under addition, multiplication (i.e.\ composition) and scalar multiplication by elements of $K$, implying that $L$ is an (associative) $K$-subalgebra of $\KK$. Moreover, if $\lambda \in L \setminus \{0\}$ then $\lambda \in \Aut_K(U_1)$ (as then $\lambda$ has a multiplicative inverse in $\KK$ and therefore $\lambda \in \Aut_{\RR}(U_1 \otimes_K \RR)$), and so $\lambda$ has a multiplicative inverse $\lambda^{-1} \in \Aut_K(U_1)$, which then clearly commutes with the action of $\hol$ and so belongs to $L$; this shows that $L$ is a division algebra over $K$. We first aim to show that the map $L \otimes_K \RR \to \KK$ is an isomorphism. Fix a $K$-basis $(e_1,\ldots,e_q)$ of $U_1$, so that $\End(U_1)$ can be identified with the ring $\Mat_q(K)$ of $q \times q$ matrices with entries in $K$.

To show that $L \otimes_K \RR \to \KK$ is surjective, we need to show that $L$ spans $\KK$ over $\RR$. Under the identification $\End(U_1 \otimes_K \RR) \cong \Mat_q(\RR)$ induced by $\End(U_1) \cong \Mat_q(K)$, the group $\hol$ is identified with a finite subgroup of $\GL[q]{K}$, and the system of equations $\{ g x = x g \mid g \in \hol \}$ over $\End(U_1 \otimes_K \RR)$ becomes a system of $\left|\hol\right| \cdot q^2$ homogeneous linear equations over $q^2$ variables with coefficients in $K$. If $S \leq \RR^{q^2}$ is the set of solutions for the latter system of equations, it then follows (from the fact that the equations are linear with coefficients in $K$) that $S \cap K^{q^2}$ is dense in $S$ and in particular spans $S$ over $\RR$. But $S$ and $S \cap K^{q^2}$ correspond to $\KK$ and $L$, respectively, implying that $L$ spans $\KK$ over $\RR$, as required.

To show that $L \otimes_K \RR \to \KK$ is injective, we need to show that if $\lambda_1,\ldots,\lambda_s \in L$ are linearly independent over $K$ then they are also linearly independent over $\RR$. Now linear independence over $K$ means that the equation $\sum_{l=1}^s x_l\lambda_l = 0$ has no non-zero solutions in $K^s$. But under the identification $\End(U_1) \cong \Mat_q(K)$, the latter equation becomes a system of $q^2$ homogeneous linear equations over $s$ variables with coefficients in $K$, implying that its solution set in $K^s$ is dense inside its solution set in $\RR^s$. Since the equation has no non-zero solutions in $K^s$ it then follows that it has no non-zero solutions in $\RR^s$, and so $\lambda_1,\ldots,\lambda_s \in L$ are linearly independent over $\RR$, as required.

Finally, let $\sigma \in \Gal(K/\QQ)$; we aim to find a $K$-subalgebra $L^\sigma \subseteq \KK$ such that the automorphism of $K$ given by $k \mapsto k^\sigma$ extends to a ring isomorphism $L \cong L^\sigma$. In order to do that, note that given $\alpha \in \End(U_1)$, the map $\alpha$ commutes with the action of $\hol$ if and only if so does $\alpha^\sigma \in \End(U_1^\sigma)$, implying that $\End_{\hol}(U_1^\sigma) = \{ \alpha^\sigma \mid \alpha \in \End_{\hol}(U_1) \}$. Now by the argument above, the map $\End_{\hol}(U_1^\sigma) \otimes_K \RR \to \End_{\hol}(U_1^\sigma \otimes_K \RR)$ is an isomorphism, implying that
\[
\dim_{\RR} \End_{\hol}(U_1^\sigma \otimes_K \RR) = \dim_K \End_{\hol}(U_1^\sigma) = \dim_K L = \dim_{\RR} \KK,
\]
and therefore $\KK \cong \End_{\hol}(U_1^\sigma \otimes_K \RR)$. Let $L^\sigma \subseteq \KK$ correspond to $\End_{\hol}(U_1^\sigma)$ under the latter isomorphism; it follows from the previous argument that $L^\sigma$ is an associative division algebra over $K$.

It remains to show that the map $L \to L^\sigma$ given by $\alpha \to \alpha^\sigma$, which is clearly a ring isomorphism, restricts to the map $k \mapsto k^\sigma$ on $K$. But by construction, an element $k \in K$ corresponds, as an element of $L$, to the scalar multiplication $\alpha_k\colon U_1 \to U_1$ sending $u \mapsto ku$. We therefore have $\alpha_k^\sigma(u^\sigma) = \alpha_k(u)^\sigma = (ku)^\sigma = k^\sigma \cdot u^\sigma$ for all $u \in U_1$, and therefore $\alpha_k^\sigma$ corresponds to the element $k^\sigma \in K \subseteq L^\sigma$, as required.
\end{proof}

In the rest of this subsection, we show that the division algebras $L^\sigma$ (for $\sigma \in \Gal(K/\QQ)$) appearing in Proposition~\ref{prop:AutU-ncryst} can be taken to be equal to $L$, and thus we have a $\Gal(K/\QQ)$-action on $L$ by ring automorphisms.  Note that if $\KK = \RR$ then we automatically have $L = K = L^\sigma$, whereas if $\KK = \CC$ then we can show that $L = L^\sigma$ using an easier argument.  The remainder of this section therefore has the most substance in the case $\KK = \HH$.  Note that in that case the construction does not automatically yield $L = L^\sigma$ (because of the existence of a continuum of $\RR$-linear automorphisms $\KK \to \KK$, unlike in the cases $\KK = \RR$ and $\KK = \CC$), but can be modified so that we indeed have $L = L^\sigma$.

\begin{lem} \label{lem:k2}
In the setting of Proposition~\ref{prop:AutU-ncryst}, let $\lambda \in L$ be such that $\kappa \coloneqq \lambda^2 \in K \cap (-\infty,0)$. Then $\frac{\kappa^\sigma}{\kappa} = k^2$ for some $k \in K$.
\end{lem}

\begin{proof}
Let $\kappa' = \sqrt{-\kappa} \in \RR$, so that the endomorphism $\lambda' \coloneqq \frac{1}{\kappa'} \lambda \in \End_{\hol}(U_1 \otimes_K \RR)$ has order $4$. Thus $\lambda'$ is diagonalisable over $\CC$ and has eigenvalues $\pm 1$ and $\pm \mathbf{i}$, with at least one eigenvalue in $\{\pm\mathbf{i}\}$; but as $\hol$ fixes the eigenspaces of $\lambda'$ and the $\hol$-representation $U_1 \otimes_K \RR$ is irreducible, it follows that the only eigenvalues of $\lambda'$ are $\pm\mathbf{i}$. This implies that we have a splitting $U_1 \otimes_K \CC = U'_+ \oplus U'_-$, where $U'_\pm$ is the $\pm\mathbf{i}$-eigenspace of $\lambda'$, and therefore the $\pm\mathbf{i}\kappa'$-eigenspace of~$\lambda$.

Now by Brauer's Theorem \cite[Theorem~10.3]{Isaacs}, the $\hol$-representation splits completely over the cyclotomic field $L'_n \coloneqq \QQ(e^{2\pi\mathbf{i}/n})$, where $n$ is the exponent of $\hol$, and by our construction we have $K \subseteq L'_n$. We thus have $U_1 \otimes_K L_n' = U''_+ \oplus U''_-$, where $U''_\pm$ are $\hol$-representations irreducible over $\CC$ and satisfying $U''_\pm \otimes_{L'_n} \CC = U'_\pm$. Given a non-zero element $u \in U_1$, we can then write $u = u_+ + u_-$ with $u_\pm \in U''_\pm$, implying that $\mathbf{i}\kappa'(u_+-u_-) = \lambda(u_+)+\lambda(u_-) = \lambda(u) \in U_1$ since $\lambda \in \End(U_1)$; but we also have $0 \neq u_+-u_- \in U_1 \otimes_K L'_n$, implying that $\mathbf{i}\kappa' \in L_n'$. Therefore, the quadratic extension $L' \coloneqq K(\mathbf{i}\kappa')$ of $K$ is contained in $L'_n$.

Now note that the extension $L_n'/\QQ$ is Galois with the Galois group $\Gal(L'_n/\QQ)$ abelian, implying that all subgroups of $\Gal(L_n'/\QQ)$ are normal. Hence the extension $L'/\QQ$ is Galois as well; this implies that $\sigma \in \Gal(K/\QQ)$ extends to an automorphism $\sigma' \in \Gal(L'/\QQ)$. Let $k \coloneqq \frac{(\mathbf{i}\kappa')^{\sigma'}}{\mathbf{i}\kappa'} \in L'$, and note that $k^2 = \frac{\kappa^\sigma}{\kappa}$ by the definition of $\kappa'$. Note also that $(\mathbf{i}\kappa')^{\sigma'} \notin \RR$ since $\mathbf{i}\kappa' \notin \RR$ and the field $K = L' \cap \RR$ is fixed by $\sigma'$, and thus $(\mathbf{i}\kappa')^{\sigma'} \in \mathbf{i}\RR$ since $[(\mathbf{i}\kappa')^{\sigma'}]^2 = \kappa^\sigma \in \RR$. This implies that $k \in \RR$, and therefore $k \in K = L' \cap \RR$, as required.
\end{proof}

\begin{cor} \label{cor:Lsigma}
In the setting of Proposition~\ref{prop:AutU-ncryst}, we may take $L^\sigma=L$.
\end{cor}

\begin{proof}
Consider the $K$-linear map $\Tr\colon L \to K$, where $\Tr(\lambda)$ is the trace of the $K$-linear map $L \to L$ sending $\mu \mapsto \lambda\mu$. Let $L_0 = \ker(\Tr)$: in particular, $L_0$ is a subspace of the $K$-vector space $L$ of codimension one. Using the facts that $L \subseteq \KK \in \{ \RR, \CC, \HH \}$ and $L \cap \RR = K$, one may verify that $L_0$ coincides with the set of elements $\lambda \in L$ such that $\lambda^2 \in K \cap (-\infty,0]$.

Now if $\KK = \RR$, then we have $L = K = L^\sigma$ and there's nothing to show. If $\KK = \CC$, then $\dim_K L_0 = 1$ implying that there exists a non-zero element $\lambda \in L_0$, and $(1,\lambda)$ is an $\RR$-basis for~$\KK$. By Lemma~\ref{lem:k2}, there exists $k \in K$ such that $\left( \frac{\lambda^\sigma}{\lambda} \right)^2 = \frac{(\lambda^2)^\sigma}{\lambda^2} = k^2$ and therefore $\lambda^\sigma = \pm k\lambda$ (note that we are working over $\KK = \CC$ here), implying that $L^\sigma = K \oplus \lambda^\sigma K = K \oplus \lambda K = L$, as required.

Suppose now that $\KK = \HH$; we aim to find injective $K$-algebra homomorphisms $\phi\colon L \to \KK$ and $\phi^\sigma\colon L^\sigma \to \KK$ such that the induced maps $L \otimes_K \RR \to \KK$ and $L^\sigma \otimes_K \RR \to \KK$ are isomorphisms and such that $\phi(L) = \phi^\sigma(L^\sigma)$. Pick a non-zero element $\lambda_1 \in L_0$. Note that for any $\lambda \in L_0$ we have $\lambda\lambda_1 + \lambda_1\lambda = (\lambda+\lambda_1)^2-\lambda^2-\lambda_1^2 \in K$, giving a $K$-linear map $\chi\colon L_0 \to K$ sending $\lambda \mapsto \lambda\lambda_1+\lambda_1\lambda$; since $\dim_K L_0 = 3$ it follows that $\chi$ is not injective, so we may pick a non-zero element $\lambda_2 \in \ker(\chi)$. Let $\lambda_3 = \lambda_1\lambda_2$; using the fact that $\lambda_1^2,\lambda_2^2 \in K \cap (-\infty,0)$ and $\lambda_1\lambda_2+\lambda_2\lambda_1 = 0$, we may check that $\lambda_3 \in L_0$ and $\chi(\lambda_3) = 0$, which then implies that the tuple $(1,\lambda_1,\lambda_2,\lambda_3)$ is $K$-linearly independent and so a $K$-basis for $L$; similarly, $(1,\lambda_1^\sigma,\lambda_2^\sigma,\lambda_3^\sigma)$ is a $K$-basis for $L^\sigma$.

Now for $i \in \{1,2,3\}$, let $c_i = \sqrt{-{\lambda_i}^2} \in \RR$ and $c_i^\sigma = \sqrt{-(\lambda_i^\sigma)^2} \in \RR$. One can then verify that the assignments $\phi(\lambda_i) = c_i\mathbf{e}_i$ and $\phi^\sigma(\lambda_i^\sigma) = c_i^\sigma\mathbf{e}_i$, where $(\mathbf{e}_1,\mathbf{e}_2,\mathbf{e}_3) = (\mathbf{i},\mathbf{j},\mathbf{k})$, extend to injective $K$-algebra homomorphisms $\phi\colon L \to \KK$ and $\phi^\sigma\colon L^\sigma \to \KK$ such that the induced maps $L \otimes_K \RR \to \KK$ and $L^\sigma \otimes_K \RR \to \KK$ are isomorphisms. Furthermore, it follows from Lemma~\ref{lem:k2} that $(\lambda_i^\sigma)^2 = k_i^2 \lambda_i^2$ and so $c_i^\sigma = \pm k_ic_i$ for some $k_i \in K$, implying that $\phi(\lambda_i K) = \phi^\sigma(\lambda_i^\sigma K)$ and therefore $\phi(L) = \phi^\sigma(L^\sigma)$, as required.
\end{proof}

\subsection{The non-crystallographic case: lattices} \label{ssec:commens-ncryst-lat}

We now describe the group $\ctlr{\Aut(\lat \cap W)}$ using Lemma~\ref{lem:ctlrK} and the description of $\ctlr{\Aut_K(T)}$ in the previous subsection.

\begin{lem} \label{lem:ncryst-commens}
Let $\Lambda$ be a complete $\mathfrak{o}_K$-lattice in $T$ and let $\Lambda' = \{ \sum_{\sigma \in \Gal(K/\QQ)} t^\sigma \mid t \in \Lambda \}$. Then $\ctlr{\Aut(\lat \cap W)}$ shares a finite-index subgroup with $\{ \widetilde\alpha|_{\Lambda'} \mid \alpha \in \ctlr{\Aut_K(\Lambda)} \}$.
\end{lem}

\begin{proof}
By definition, $\ctlr{\Aut(\lat \cap W)} = \ctlr{\Aut(W)} \cap \Aut(\lat \cap W)$. On the other hand, we claim that $\Lambda'$ is a lattice in $W$ and that $\{ \widetilde\alpha|_{\Lambda'} \mid \alpha \in \ctlr{\Aut_K(\Lambda)} \} = \ctlr{\Aut(W)} \cap \Aut(\Lambda')$. In that case, $\Aut(\lat \cap W)$ and $\Aut(\Lambda')$ share a finite-index subgroup by Lemma~\ref{lem:Aut-lat-fi}, implying the result.

To see that $\Lambda'$ is a lattice (i.e.\ a complete $\ZZ$-lattice in $W$), note that every $x \in K$ can be expressed as $x=a/b$ for some $a \in \mathfrak{o}_K$ and $b \in \ZZ \setminus \{0\}$. Thus any $x' \in T$ is equal to $a'/b$ for some $a' \in \Lambda$ and $b \in \ZZ \setminus \{0\}$, implying that $W = \{ \sum_{\sigma \in \Gal(K/\QQ)} t^\sigma \mid t \in T \}$ is spanned by $\Lambda'$ over~$\QQ$. Since $\Lambda$ is finitely generated over $\mathfrak{o}_K$ and $\mathfrak{o}_K$ is finitely generated over $\ZZ$, it also follows that $\Lambda'$ is finitely generated as an abelian group; hence $\Lambda'$ is a lattice, as required.

To see that $\{ \widetilde\alpha|_{\Lambda'} \mid \alpha \in \ctlr{\Aut_K(\Lambda)} \} = \ctlr{\Aut(W)} \cap \Aut(\Lambda')$, note that the inclusion $\subseteq$ follows from Lemma~\ref{lem:ctlrK} and the definition of $\widetilde\alpha$. Conversely, suppose $\widetilde\alpha$ fixes $\Lambda'$ for some $\alpha \in \ctlr{\Aut_K(T)}$. Then $\sum_{\sigma \in \Gal(K/\QQ)} \alpha^{\pm 1}(t)^\sigma = \widetilde\alpha^{\pm 1}(\sum_{\sigma \in \Gal(K/\QQ)} t^\sigma) \in \widetilde\alpha^{\pm 1}(\Lambda') = \Lambda'$ for all $t \in \Lambda$, implying that $\alpha^{\pm 1}(t) \in \Lambda$. We thus have $\alpha \in \Aut_K(\Lambda)$, as required.
\end{proof}

\begin{thm} \label{thm:commens-1factor}
Suppose $U_1 \otimes_K \RR$ (and so each $U_j \otimes_K \RR$) is an irreducible $\RR$-representation of type $\KK \in \{ \RR, \CC, \HH \}$. Then, under the isomorphism $\ctlr{\Aut_\RR(W \otimes \RR)} \cong \GL[m]{\KK}^{[K:\QQ]}$, a finite-index subgroup of $\ctlr{\Aut(\lat \cap W)}$ is identified with a finite-index subgroup of the group
\[
\{ (A^\sigma \mid \sigma \in \Gal(K/\QQ)) \mid A \in \GL[m]{\mathfrak{o}} \},
\]
where $\mathfrak{o}$ is an $\mathfrak{o}_K$-order in a division $K$-subalgebra $L \subseteq \KK$ satisfying the following properties:
\begin{itemize}
    \item the canonical map $L \otimes_K \RR \to \KK$ is an isomorphism,
    \item the $\Gal(K/\QQ)$-action on $K$ extends to an action on $L$ by ring automorphisms, and
    \item the $\Gal(K/\QQ)$-action on $L$ restricts to an action on $\mathfrak{o}$.
\end{itemize}
\end{thm}

\begin{proof}
Under the identification $\ctlr{\Aut_\RR(W \otimes \RR)} \cong \GL[m]{\KK}^r$, the group $\ctlr{\Aut_\RR(T \otimes_K \RR)}$ corresponds to the first copy of $\GL[m]{\KK}$, with its subgroup $\ctlr{\Aut_K(T)}$ identified, by Proposition~\ref{prop:AutU-ncryst}, with $\GL[m]{L}$ for a division $\KK$-subalgebra $L$, which satisfies the required properties by Proposition~\ref{prop:AutU-ncryst} and Corollary~\ref{cor:Lsigma}. We first aim to find a complete $\mathfrak{o}_K$-lattice $\Lambda \subseteq T$ such that $\ctlr{\Aut_K(\Lambda)}$ corresponds to $\GL[m]{\overline{\mathfrak{o}}}$ under this identification, for some $\mathfrak{o}_K$-order $\overline{\mathfrak{o}}$ in $L$, then replace $\overline{\mathfrak{o}}$ with a $\Gal(K/\QQ)$-invariant $\mathfrak{o}_K$-suborder $\mathfrak{o} \subseteq \overline{\mathfrak{o}}$ and show that $\GL[m]{\mathfrak{o}}$ has finite index in $\GL[m]{\overline{\mathfrak{o}}}$. The result then follows by Lemma~\ref{lem:ncryst-commens}.

In order to define $\Lambda$, let $\Lambda''$ be any complete $\mathfrak{o}_K$-lattice in $U_1$ (e.g.\ $\Lambda'' = (\lat \otimes \mathfrak{o}_K) \cap U_1$). Let $\phi_j\colon U_1 \to U_j$ be an $\hol$-equivariant $K$-linear isomorphism for $1 \leq j \leq m$, and set $\Lambda = \bigoplus_{j=1}^m \phi_j(\Lambda'')$. It is then clear by construction that $\ctlr{\Aut_K(\Lambda)} \cong \GL[m]{\overline{\mathfrak{o}}}$ where $\overline{\mathfrak{o}} \subseteq L$ corresponds to $\End_{\hol}(\Lambda'')$ under the isomorphism $L \cong \End_{\hol}(U_1)$. It is clear that $\overline{\mathfrak{o}}$ is a subring of $L$, and it is an $\mathfrak{o}_K$-submodule of $L$ since $\Lambda''$ is an $\mathfrak{o}_K$-lattice. To see that $\overline{\mathfrak{o}}$ spans $L$ over $K$, it is enough to notice that $\Lambda''$ is finitely generated over $\mathfrak{o}_K$ and spans $U_1$ over $K$, implying that for any $\lambda \in L$ there exists $k \in K$ such that $k\lambda \in \overline{\mathfrak{o}}$. Finally, $\overline{\mathfrak{o}}$ is finitely generated over $\mathfrak{o}_K$ since it is a submodule of $\End(\Lambda'')$ (which is finitely generated since $\Lambda''$ is finitely generated over $\mathfrak{o}_K$) and since the ring $\mathfrak{o}_K$ is Noetherian. Thus $\overline{\mathfrak{o}}$ is an $\mathfrak{o}_K$-order in $L$, as required.

Note that since $\overline{\mathfrak{o}} \subset L$ is an $\mathfrak{o}_K$-order, so is $\overline{\mathfrak{o}}^\sigma \coloneqq \{ \lambda^\sigma \mid \lambda \in \overline{\mathfrak{o}} \}$ for any $\sigma \in \Gal(K/\QQ)$.  Let $\mathfrak{o} = \bigcap_{\sigma \in \Gal(K/\QQ)} \overline{\mathfrak{o}}^\sigma$.  Since each $\overline{\mathfrak{o}}^\sigma$ is a complete $\ZZ$-lattice in the $\QQ$-vector space $L$, so is $\mathfrak{o}$, and in particular $\overline{\mathfrak{o}}$ and $\mathfrak{o}$ have the same rank as abelian groups, and therefore as $\mathfrak{o}_K$-modules.  It follows that $\mathfrak{o}$ is a complete $\mathfrak{o}_K$-lattice; on the other hand, it is clearly a subring of $L$ and thus an $\mathfrak{o}_K$-order.

Finally, we claim that $\GL[m]{\mathfrak{o}}$ has finite index in $\GL[m]{\overline{\mathfrak{o}}}$.  Indeed, note that by looking at $\GL[m]{L}$ as the group of automorphisms of the free $L$-module $L^m$, the subgroups $\GL[m]{\overline{\mathfrak{o}}}$ and $\GL[m]{\mathfrak{o}}$ correspond to setwise stabilisers of the $\mathfrak{o}_K$-lattices $\overline{\mathfrak{o}}^m$ and $\mathfrak{o}^m$, respectively. But by the previous paragraph, $\mathfrak{o}$ has index $d < \infty$ in $\overline{\mathfrak{o}}$, implying that $\mathfrak{o}^m$ has index $d^m$ in $\overline{\mathfrak{o}}^m$.  The group $\GL[m]{\overline{\mathfrak{o}}}$ permutes the $\mathfrak{o}_K$-sublattices of $\overline{\mathfrak{o}}^m$ of index $d^m$, so as there are finitely many such lattices, it follows that $\GL[m]{\mathfrak{o}}$ has finite index in $\GL[m]{\overline{\mathfrak{o}}}$, as claimed.
\end{proof}

\begin{proof}[Proof of Theorem~\ref{thm:commensurable}]
    This follows directly from Corollary~\ref{cor:commens} and Theorem~\ref{thm:commens-1factor}.
\end{proof}

\begin{proof}[Proof of Corollary~\ref{cor:commensurable-orb}]
    This follows directly from Proposition~\ref{prop:Mflat-orb} and Theorem~\ref{thm:commensurable}, together with Theorem~\ref{thm:Bettiol-Derdzinski-Piccione} and the surrounding discussion.
\end{proof}

\begin{rmk}
One may show, even without Standing Assumption~\ref{sta:Schur1}, that any irreducible constituent of the $\hol$-representation $\RR^n$ is realizable over an abelian number field that is a subfield of $\RR$.  Indeed, let $C = \QQ(e^{2\pi\mathbf{i}/o})$, where $o$ is the exponent of $\hol$, and let $U_j' \leq V_j \otimes \RR$ be an irreducible $\hol$-subspace in the above notation. By Brauer's Theorem \cite[Theorem~10.3]{Isaacs}, any representation of $\hol$ splits completely over $C$: that is, if $T$ is an irreducible $\hol$-representation over $C$ then $T \otimes_C \CC$ is irreducible over~$\CC$. In particular, $U_j'$ is realizable over $C$; by \cite[Theorem~1.3]{Pasechnik}, any representation over $C$ that is realizable over $\RR$ is also realizable over $C \cap \RR$, implying that $U_j'$ is realizable over $C \cap \RR$.
\end{rmk}

\bibliographystyle{amsalpha}
\bibliography{ref}

\end{document}